\documentclass[reqno]{amsart}
\usepackage{float}
\usepackage{amsmath}
\usepackage{graphicx}
\usepackage{latexsym}
\usepackage{amsfonts}
\usepackage{amssymb}

\theoremstyle{plain}
\newtheorem{theorem}{Theorem}
\newtheorem{corollary}[theorem]{Corollary}
\newtheorem{lemma}[theorem]{Lemma}

\theoremstyle{definition}

\newtheorem{remark}[theorem]{Remark}
\newtheorem*{remark*}{Remark}

\begin{document}
\title{Local probabilities for random walks conditioned to stay positive}
\thanks{Supported in part by the Russian Foundation for Basic Research grant 08-01-00078}
\author[Vatutin]{Vladimir A. Vatutin}
\address{Steklov Mathematical Institute RAS, Gubkin street 8, 19991, Moscow
\\
Russia}
\email{vatutin@mi.ras.ru}
\author[Wachtel]{Vitali Wachtel}
\address{Technische Universit{\"a}t M{\"u}nchen, Zentrum Mathematik, Bereich
M5, D-85747 Garching bei M{\"u}nchen}
\email{vakhtel@ma.tum.de}
\date{\today }

\begin{abstract}
Let $S_{0}=0,\{S_{n},\,n\geq 1\}$ be a random walk generated by a sequence
of i.i.d. random variables $X_{1},X_{2},...$ and let $\tau ^{-}=\min \{n\geq
1:S_{n}\leq 0\}$ and $\tau ^{+}=\min \{n\geq 1:S_{n}>0\}$. Assuming that the
distribution of $X_{1}$ belongs to the domain of attraction of an $\alpha $%
-stable law we study the asymptotic behavior, as $n\rightarrow \infty $, of
the local probabilities $\mathbf{P}(\tau ^{\pm }=n)$ and the conditional
local probabilities $\mathbf{P}(S_{n}\in \lbrack x,x+\Delta )|\tau ^{-}>n)$
for fixed $\Delta $ and $x=x(n)\in (0,\infty )$ $.$
\end{abstract}

\maketitle

\section{Introduction and main result}

Let $S_{0}:=0$, $S_{n}:=X_{1}+\ldots +X_{n},$ $n\geq 1$, be a random walk
where the $X_{i}$ are independent copies of a random variable $X$ and
\begin{equation*}
\tau ^{-}=\min \left\{ n\geq 1:\ S_{n}\leq 0\right\} \text{ \ and \ \ }\tau
^{+}=\min \left\{ n\geq 1:\ S_{n}>0\right\}
\end{equation*}%
be the first weak descending and first strict ascending ladder epochs of $%
\left\{ S_{n},\,n\geq 0\right\} $. \ The aim of this paper is to study, as $%
n\rightarrow \infty ,$ the asymptotic behavior of the local probabilities $%
\mathbf{P}\left( \tau ^{\pm }=n\right) $ and conditional local probabilities
$\mathbf{P}(S_{n}\in \lbrack x,x+\Delta )|\tau ^{-}>n)$ for fixed $\Delta >0$
and $x=x(n)\in \left( 0,\infty \right) $.

To formulate our results we let
\begin{equation*}
\mathcal{A}:=\{0<\alpha <1;\,|\beta |<1\}\cup \{1<\alpha <2;|\beta |\leq
1\}\cup \{\alpha =1,\beta =0\}\cup \{\alpha =2,\beta =0\}
\end{equation*}%
be a subset in $\mathbb{R}^{2}.$ For $(\alpha ,\beta )\in \mathcal{A}$ and a
random variable $X$ write $X\in \mathcal{D}\left( \alpha ,\beta \right) $ if
the distribution of $X$ belongs to the domain of attraction of a stable law
with characteristic function%
\begin{equation}
G_{\alpha ,\beta }\mathbb{(}t\mathbb{)}:=\exp \left\{ -c|t|^{\,\alpha
}\left( 1-i\beta \frac{t}{|t|}\tan \frac{\pi \alpha }{2}\right) \right\}
=\int_{-\infty }^{+\infty }e^{itx}g_{\alpha ,\beta }(u)du,\ c>0,  \label{std}
\end{equation}%
and, in addition, $\mathbf{E}X=0$ if this moment exists.

\ Denote $\mathbb{Z}:=\left\{ 0,\pm 1,\pm 2,..\right\} ,$ $\mathbb{Z}%
_{+}:=\left\{ 1,2,..\right\} $ and let $\left\{ c_{n},n\geq 1\right\} $ be a
sequence of positive integers specified by the relation%
\begin{equation}
c_{n}:=\inf \left\{ u\geq 0:\mu (u)\leq n^{-1}\right\} ,  \label{Defa}
\end{equation}%
where
\begin{equation*}
\mu (u):=\frac{1}{u^{2}}\int_{-u}^{u}x^{2}\mathbf{P}(X\in dx).
\end{equation*}%
It is known (see, for instance, \cite[Ch. XVII, \S 5]{FE}) that for every $%
X\in \mathcal{D}(\alpha ,\beta )$ the function $\mu (u)$ is regularly
varying with index $(-\alpha )$. This implies that $\left\{ c_{n},n\geq
1\right\} $ is a regularly varying sequence with index $\alpha ^{-1}$, i.e.,
there exists a function $l_{1}(n),$ slowly varying at infinity, such that
\begin{equation}
c_{n}=n^{1/\alpha }l_{1}(n).  \label{asyma}
\end{equation}%
In addition, the scaled sequence $\left\{ S_{n}/c_{n},\,n\geq
1\right\} $ converges in distribution, as $n~\rightarrow~\infty ,$
to the stable law given by (\ref{std}).

The following conditional limit theorem will be crucial for the rest of this
article.

\begin{theorem}
\label{Cru} If $X\in \mathcal{D}(\alpha ,\beta )$, then there exists a
nonnegative random variable $M_{\alpha ,\beta }$ with density $p_{\alpha
,\beta }(u)$ such that, for all $u_{2}>u_{1}\geq 0$,
\begin{equation}
\lim_{n\rightarrow \infty }\mathbf{P}\left( \frac{S_{n}}{c_{n}}\in \lbrack
u_{1},u_{2})\ \Big|\ \tau ^{-}>n\right) =\mathbf{P}(M_{\alpha ,\beta }\in
\lbrack u_{1},u_{2}))=\int_{u_{1}}^{u_{2}}p_{\alpha ,\beta }(v)dv.
\label{meander1}
\end{equation}
\end{theorem}

The validity of the first equality in (\ref{meander1}) was established by
Durrett \cite{Dur78} and we \ believe that the absolutely continuity of the
distribution of $M_{\alpha ,\beta }$ is also known in the literature, but
failed to find any reference. However, this fact will be a by-product of our
arguments and we include it in (\ref{meander1}) to simplify the statements
of the main theorems of the present paper.

Our first result is an analog of the classical Stone local limit theorem.

\begin{theorem}
\label{NormalDev'} Suppose $X\in \mathcal{D}(\alpha ,\beta )$ and the
distribution of $X$ is non-lattice. Then, for every $\Delta >0$,
\begin{equation}
c_{n}\mathbf{P}(S_{n}\in \lbrack x,x+\Delta )|\tau ^{-}>n)-\Delta {p_{\alpha
,\beta }(x/c_{n})}\rightarrow 0\text{ as }n\rightarrow \infty  \label{ND'}
\end{equation}%
uniformly in $x\in \left( 0,\infty \right) $.
\end{theorem}

For the case when the distribution of $X$ belongs to the domain of
attraction of the normal law, that is, when $X\in \mathcal{D}(2,0)$ relation
(\ref{ND'}) \ has been proved by Caravenna \cite{Car05}.

If the ratio $x/c_{n}$ varies with $n$ in such a way that $x/c_{n}\in
(b_{1},b_{2})$ for some $0<b_{1}<b_{2}<\infty $, we can rewrite (\ref{ND'})
as
\begin{equation*}
c_{n}\mathbf{P}(S_{n}\in \lbrack x,x+\Delta )|\tau ^{-}>n)\sim \Delta {%
p_{\alpha ,\beta }(x/c_{n})}\quad \text{as }n\rightarrow \infty .
\end{equation*}%
However, if $x/c_{n}\rightarrow 0$, then, in view of
\begin{equation*}
\lim_{z\downarrow 0}{p_{\alpha ,\beta }(z)=0}
\end{equation*}%
(see (\ref{75}) below), relation (\ref{ND'}) gives only
\begin{equation}
c_{n}\mathbf{P}(S_{n}\in \lbrack x,x+\Delta )|\tau ^{-}>n)=o\left( 1\right)
\quad \text{as }n\rightarrow \infty .  \label{Out1}
\end{equation}

Our next theorem refines (\ref{Out1}) in the mentioned domain of small
deviations, i.e., when $x/c_{n}\rightarrow 0.$ To formulate the desired
statement we need some additional notation.

Set $\chi ^{+}:=S_{\tau ^{+}}$ and introduce the renewal function
\begin{equation}
H(u):=\mathrm{I}\{u>0\}+\sum_{k=1}^{\infty }\mathbf{P}(\chi _{1}^{+}+\ldots
+\chi _{k}^{+}<u).  \label{ren}
\end{equation}%
Clearly, $H$ is a left-continuous function.

\begin{theorem}
\label{SmallDev'} Suppose $X\in \mathcal{D}(\alpha ,\beta )$ and the
distribution of $X$ is non-lattice. Then
\begin{equation}
c_{n}\mathbf{P}(S_{n}\in \lbrack x,x+\Delta )|\tau ^{-}>n)\sim g_{\alpha
,\beta }(0)\frac{\int_{x}^{x+\Delta }H(u)du}{n\mathbf{P}\left( \tau
^{-}>n\right) }\quad \text{as }n\rightarrow \infty  \label{SD'}
\end{equation}%
uniformly in $x\in (0,\delta _{n}c_{n}]$, where $\delta _{n}\rightarrow 0$
as $n\rightarrow \infty $.
\end{theorem}

We continue by considering the lattice case and say that a random variable $%
X $ is $(h,a)-$lattice if the distribution of $X$ is lattice with span $h>0$
and shift $a\in \lbrack 0,h)$, i.e., the $h$ is the maximal number such that
the support of the distribution of $X$ is contained in the set $\left\{ a+kh,%
\text{ }k=0,\pm 1,\pm 2,...\right\} $.

\begin{theorem}
\label{NormalDev} Suppose $X\in \mathcal{D}(\alpha ,\beta )$ and is $(h,a)- $%
lattice. Then
\begin{equation}
c_{n}\mathbf{P}(S_{n}=an+x|\tau ^{-}>n)-h{p_{\alpha ,\beta }((an+x)/c_{n})}%
\rightarrow 0\text{ as }n\rightarrow \infty  \label{ND}
\end{equation}%
uniformly in $x\in \left( -an,\infty \right) \cap h\mathbb{Z}$.
\end{theorem}

For $X\in \mathcal{D}(2,0)$ and being $\left( h,0\right) -$lattice relation (%
\ref{ND}) has been obtained by Bryn-Jones and Doney \cite{BJD06}.

\begin{theorem}
\label{SmallDev} Suppose $X\in \mathcal{D}(\alpha ,\beta )$ and is $(h,a)-$%
lattice. Then
\begin{equation}
c_{n}\mathbf{P}(S_{n}=an+x|\tau ^{-}>n)\sim hg_{\alpha ,\beta }(0)\frac{%
H(an+x)}{n\mathbf{P}\left( \tau ^{-}>n\right) }\quad \text{as }n\rightarrow
\infty  \label{SD}
\end{equation}%
uniformly in $x\in (-an,-an+\delta _{n}c_{n}]\cap h\mathbb{Z}$, where $%
\delta _{n}\rightarrow 0$ as $n\rightarrow \infty .$
\end{theorem}

Note that Alili and Doney \cite{AD99} established (\ref{SD}) under the
assumptions $X$ is $(h,0)-$lattice and $\mathbf{E}X^{2}<\infty $. Bryn-Jones
and Doney \cite{BJD06} generalized their results to the $(h,0)-$lattice $%
X\in \mathcal{D}(2,0)$.

The next theorem describes the asymptotic behavior of the density function $%
p_{\alpha ,\beta }$ at zero. The explicit form of $p_{\alpha ,\beta }$ is
known only for ${\alpha =2,\beta =0}$: $p_{2,0}(x)=xe^{-x^{2}/2}\mathrm{I}%
(x>0)$. \ For this reason we deduce an integral equation for $p_{\alpha
,\beta }$ (see (\ref{C0}) below) and using Theorems \ref{NormalDev'}-\ref%
{SmallDev} find the asymptotic behavior of $p_{\alpha ,\beta }(z)$ at zero.

\begin{theorem}
\label{Density} For every $(\alpha ,\beta )\in \mathcal{A}$ there exists a
constant $C>0$ such that
\begin{equation*}
p_{\alpha ,\beta }(z)\sim Cz^{\alpha \rho }\text{ as }z\downarrow 0,
\end{equation*}%
where $\rho :=\int_{0+}^{\infty }g_{\alpha ,\beta }(u)du$.
\end{theorem}

One of our main motivations to be interested in the local probabilities of
conditioned random walks is the question on the asymptotic behavior of the
local probabilities of the ladder epochs $\tau ^{-}$ and $\tau ^{+}$. Before
formulating the relevant results we recall some known facts concerning the
properties of these random variables given

\begin{equation*}
\sum_{n=1}^{\infty }\frac{1}{n}\mathbf{P}(S_{n}>0)=\sum_{n=1}^{\infty }\frac{%
1}{n}\mathbf{P}(S_{n}\leq 0)=\infty .
\end{equation*}%
The last means that $\{S_{n},n\geq 0\}$ is an oscillating random walk, and,
in particular, the stopping moments $\tau ^{-}$ and \ \ $\tau ^{+}$ are
well-defined proper random variables. Moreover, it follows from the
Wiener-Hopf factorization (see, for example, \cite[Theorem~8.9.1,~p.~376]%
{BGT}) that for all $z\in \lbrack 0,1)$,
\begin{equation}
1-\mathbf{E}z^{\tau ^{-}}=\exp \left\{ -\sum_{n=1}^{\infty }\frac{z^{n}}{n}%
\mathbf{P}(S_{n}\leq 0)\right\}  \label{weak}
\end{equation}%
and
\begin{equation}
1-\mathbf{E}z^{\tau ^{+}}=\exp \left\{ -\sum_{n=1}^{\infty }\frac{z^{n}}{n}%
\mathbf{P}(S_{n}>0)\right\} .  \label{strict}
\end{equation}%
Rogozin \cite{Rog71} investigated properties of $\tau ^{+}$ and demonstrated
that the Spitzer condition
\begin{equation}
n^{-1}\sum_{k=1}^{n}\mathbf{P}\left( S_{k}>0\right) \rightarrow \rho \in
\left( 0,1\right) \quad \text{as }n\rightarrow \infty  \label{Spit}
\end{equation}%
holds if and only if $\tau ^{+}$ belongs to the domain of attraction of a
positive stable law with parameter $\rho $. In particular, if $X\in \mathcal{D}%
(\alpha ,\beta )$ then (see,$\ $\ for instance, \cite{Zol57}) condition (\ref%
{Spit}) holds with
\begin{equation}
\displaystyle\rho =\int_{0+}^{\infty }g_{\alpha ,\beta }(u)du=\left\{
\begin{array}{ll}
\frac{1}{2},\ \alpha =1, &  \\
\frac{1}{2}+\frac{1}{\pi \alpha }\arctan \left( \beta \tan \frac{\pi \alpha
}{2}\right) ,\text{ otherwise}. &
\end{array}%
\right.  \label{ro}
\end{equation}

Since (\ref{weak}) and (\ref{strict}) imply
\begin{equation*}
(1-\mathbf{E}z^{\tau ^{+}})(1-\mathbf{E}z^{\tau ^{-}})=1-z\quad \text{for
all }z\in (0,1),
\end{equation*}%
one can deduce by Rogozin's result that (\ref{Spit}) holds if and only if
there exists a function $l(n)$ slowly varying at infinity such that, as $%
n\rightarrow \infty $,
\begin{equation}
\mathbf{P}\left( \tau ^{-}>n\right) \sim \frac{l(n)}{n^{1-\rho }},\ \
\mathbf{P}\left( \tau ^{+}>n\right) \sim \frac{1}{\Gamma (\rho )\Gamma
(1-\rho )n^{\rho }l(n)}.  \label{integ1}
\end{equation}%
We also would like to mention that, according to Doney \cite{Don95}, the
Spitzer condition is equivalent to
\begin{equation}
\mathbf{P}\left( S_{n}>0\right) \rightarrow \rho \in \left( 0,1\right) \quad
\text{as }n\rightarrow \infty .  \label{SpitDon}
\end{equation}%
Therefore, both relations in (\ref{integ1}) are valid under condition (\ref%
{SpitDon}).

The asymptotic representations (\ref{integ1}) include a slowly varying
function $l(x)$ which is of interest as well. Unfortunately, to get a more
detailed information about the asymptotic properties of $l(x)$ it is
necessary to impose additional hypotheses on the distribution of $X$. Thus,
Rogozin \cite{Rog71} has shown that $l(x)$ is asymptotically a constant if
and only if
\begin{equation}
\sum_{n=1}^{\infty }\frac{1}{n}\Bigl(\mathbf{P}\left( S_{n}>0\right) -\rho %
\Bigr)<\infty .  \label{Rog}
\end{equation}%
It follows from the Spitzer-R\'{o}sen theorem (see \cite[%
Theorem~8.9.23,~p.~382]{BGT}) that if $\mathbf{E}X^{2}<\infty $, then (\ref%
{Rog}) holds with $\rho =1/2$, and, consequently,
\begin{equation}
\mathbf{P}(\tau ^{\pm }>n)\sim \frac{C^{\pm }}{n^{1/2}}\quad \text{as }%
n\rightarrow \infty ,  \label{sym}
\end{equation}%
where $C^{\pm }$ are positive constants. Much less is known about the form
of $l(x)$ if $\mathbf{E}X^{2}=\infty .$ For instance, if the distribution of
$X$ is symmetric, then, clearly,
\begin{equation}
\left\vert \mathbf{P}\left( S_{n}>0\right) -\frac{1}{2}\right\vert =\frac{1}{%
2}\mathbf{P}\left( S_{n}=0\right) .  \label{simm}
\end{equation}%
Furthermore, according to \cite[Theorem III.9, p.~49]{Pet75}, there exists $%
C>0$ such that for all $n\geq 1$,
\begin{equation*}
\mathbf{P}\left( S_{n}=0\right) \leq \frac{C}{\sqrt{n}}.
\end{equation*}%
By this estimate and (\ref{simm}) we conclude that (\ref{Rog}) holds with $%
\rho =1/2$ and, therefore, (\ref{sym}) is valid for all symmetric random
walks.

One more situation was analyzed by Doney \cite{Don82a}. \ Assuming that $%
\mathbf{P}(X>x)=\left( x^{\alpha }l_{0}(x)\right) ^{-1},\ x>0,$ with $%
1<\alpha <2$ and $l_{0}(x)$ slowly varying at infinity, he established some
relationships between the asymptotic behavior of $l_{0}(x)$ and $l(x)$ at
infinity for a number of cases.

Thus, up to now \ there is a group of results describing the behavior of the
probabilities $\mathbf{P}(\tau ^{\pm }>n)$ as $n\rightarrow \infty $ and the
functions involved in their asymptotic representations. \ We complement the
mentioned statements by the following two theorems describing the behavior
of the local probabilities $\mathbf{P}(\tau ^{\pm }=n)$ as $n\rightarrow
\infty $. \

\begin{theorem}
\label{LadEpoch} If $X\in \mathcal{D}(\alpha ,\beta )$ then there exists a
sequence $\{Q_{n}^{-},n\geq 1\}$ such that
\begin{equation}
\mathbf{P}(\tau ^{-}=n)=Q_{n}^{-}\frac{l(n)}{n^{2-\rho }}(1+o(1))\text{ as }%
n\rightarrow \infty .  \label{LE}
\end{equation}%
The sequence $\{Q_{n}^{-},n\geq 1\}$ is bounded from above, and there exists
a positive constant $Q_{\ast }^{-}$ such that $Q_{n}^{-}\mathrm{I}%
(Q_{n}^{-}>0)\geq Q_{\ast }^{-}$ for all $n\geq 1$. \ Moreover, we may
choose $Q_{n}^{-}\equiv 1-\rho $ if and only if one of the following
conditions holds:

\begin{itemize}
\item[(a)] $\mathbf{E}(-S_{\tau^-})=\infty$,

\item[(b)] $\mathbf{E}(-S_{\tau ^{-}})<\infty $ and the distribution of $X$
is $(h,0)-$lattice,

\item[(c)] $\mathbf{E}(-S_{\tau^-})<\infty$ and the distribution of $X$ is
non-lattice.
\end{itemize}
\end{theorem}

\begin{remark}
The statement of the theorem includes the quantity $\mathbf{E}(-S_{\tau
^{-}})$, which depends on $\tau ^{-}$, \ a random variable being the
objective of the theorem. This is done only to simplify the form of the
theorem. In fact, Chow \cite{Chow} has shown that $\mathbf{E}(-S_{\tau
^{-}}) $ is finite if and only if
\begin{equation*}
\int_{0}^{\infty }\frac{x^{2}}{\int_{0}^{\infty }y\min \{x,y\}\mathbf{P}%
(X^{+}\in dy)}\mathbf{P}(X^{-}\in dx)<\infty ,
\end{equation*}%
where $X^{+}:=\max \{0,X\}$ and $X^{-}:=-\min \{0,X\}$.
\end{remark}

\begin{remark}
The simple random walk in which $P(X=\pm 1)=1/2$ is the most natural example
with $Q_{n}^{-}\neq 1-\rho $. Here $P(\tau ^{-}=2k+1)=0$ and, consequently, $%
Q_{2k+1}^{-}=0$ for all $k\geq 1$. On the other hand, $\lim_{k\rightarrow
\infty }Q_{2k}^{-}$ exists and is strictly positive. This result is in
complete agreement with Theorem~\ref{LadEpoch}: the step-distribution of the
simple random walk is $(2,1)$-lattice.
\end{remark}

For the stopping time $\tau ^{+}$ we have a similar statement:

\begin{theorem}
\label{LadEpoch'} If  $X\in \mathcal{D}(\alpha ,\beta )$ then there exists a
sequence $\{Q_{n}^{+},n\geq 1\}$ such that
\begin{equation}
\mathbf{P}(\tau ^{+}=n)=Q_{n}^{+}\frac{l(n)}{n^{1+\rho }}(1+o(1))\text{ \ as
}n\rightarrow \infty .  \label{LE'}
\end{equation}%
The sequence $\{Q_{n}^{+},n\geq 1\}$ is bounded from above, and there exists
a positive constant $Q_{\ast }^{+}$ such that $Q_{n}^{+}\mathrm{I}%
(Q_{n}^{+}>0)\geq Q_{\ast }^{+}$ for all $n\geq 1$. Moreover, we may choose $%
Q_{n}^{+}\equiv \rho /(\Gamma (\rho )\Gamma (1-\rho ))$ if and only if one
of the following conditions holds:

\begin{itemize}
\item[(a)] $\mathbf{E}S_{\tau^+}=\infty$,

\item[(b)] $\mathbf{E}S_{\tau ^{+}}<\infty $ and the distribution of $X$ is $%
(h,0)-$lattice,

\item[(c)] $\mathbf{E}S_{\tau ^{+}}<\infty $ and the distribution of $X$ is
non-lattice.
\end{itemize}
\end{theorem}

In some special cases the asymptotic behavior of $\mathbf{P}\left( \tau
^{\pm }=n\right) $ is already known from the literature. Eppel \cite{Epp}
proved that if $\mathbf{E}X=0$, $\mathbf{E}X^{2}$ is finite, and the
distribution of $X$ is non-lattice, then
\begin{equation}
\mathbf{P}\left( \tau ^{\pm }=n\right) \sim \frac{C^{\pm }}{n^{3/2}}\text{ \
as }n\rightarrow \infty .  \label{integ2}
\end{equation}%
Clearly, $X\in \mathcal{D}(2,0)$ in this case. For aperiodic random walks on
integers with $EX=0$ and $EX^{2}<\infty $ representations (\ref{integ2})
were obtained by Alili and Doney \cite{AD99}.

Asymptotic relation (\ref{integ2}) is valid for all continuous symmetric
(implying $\rho =1/2$ in (\ref{SpitDon})) random walks (see \cite[Chapter
XII, Section 7]{FE}). Note that the restriction $X\in \mathcal{D}(\alpha
,\beta )$ is superfluous in this situation.

Recently A.Borovkov \cite{Bor04} has shown that if (\ref{Spit}) is valid and
\begin{equation}
n^{1-\rho }\Bigl(\mathbf{P}(S_{n}>0)-\rho \Bigr)\rightarrow const\in
(-\infty ,\infty )\quad \text{as }n\rightarrow \infty ,  \label{BorCond}
\end{equation}%
then (\ref{LE}) holds with $l(n)\equiv const\in (0,\infty )$. Proving the
mentioned result Borovkov does not assume that the distribution of $X$ is
taken from the domain of attraction of a stable law. However, he gives no
explanations how one can check the validity of (\ref{BorCond}) in the
general situation.

Further, Alili and Doney \cite[Remark~1,~p.~98]{AD99} have demonstrated that
if $X$ is $(h,0)-$ lattice and $\mathbf{E}S_{\tau ^{+}}<\infty $ then (\ref%
{LE'}) holds with $Q_{n}^{+}\sim \rho /(\Gamma (\rho )\Gamma (1-\rho ))$.

Finally, Mogulski and Rogozin \cite{MR05} established (\ref{LE}) for $X$
satisfying the conditions $\mathbf{E}X=0$ and $\mathbf{E}|X|^{3}<\infty .$
Moreover, they proved that $Q_{n}^{+}\sim const$ if and only if the
distribution of $X$ is either non-lattice or $(h,0)-$ lattice. Observe that $%
\mathbf{E}(-S_{\tau ^{-}})<\infty $ under their conditions.

\section{Auxiliary results}

\subsection{Notation}

In what follows we denote by $C,C_{1},C_{2},...$ finite positive constants
which may be \textit{different} from formula to formula and by $%
l(x),l_{0}(x),l_{1}(x),l_{2}(x),...$ functions slowly varying at infinity
which are, as a rule, \textit{fixed} \textit{once and forever}.

It is known that if $X\in \mathcal{D}\left( \alpha ,\beta \right) $ with $%
\alpha \in (0,2),$ and $F(x):=\mathbf{P}\left( X<x\right) $, then
\begin{equation}
1-F(x)+F(-x)\sim \frac{1}{x^{\alpha }l_{0}(x)}\quad \text{as }x\rightarrow
\infty ,  \label{Tailtwo}
\end{equation}%
where $l_{0}(x)$ is a function slowly varying at infinity. Besides, for $%
\alpha \in (0,2)$,%
\begin{equation}
\frac{F(-x)}{1-F(x)+F(-x)}\rightarrow q,\quad \frac{1-F(x)}{1-F(x)+F(-x)}%
\rightarrow p\quad \text{as }x\rightarrow \infty ,  \label{tailF}
\end{equation}%
with $p+q=1$ and $\beta =p-q$ in (\ref{std}). It is easy to see that (\ref%
{Tailtwo}) implies
\begin{equation}
\mu (u)\sim \frac{\alpha }{2-\alpha }\mathbf{P}(|X|>u)\quad \text{as }%
u\rightarrow \infty .  \label{UU}
\end{equation}%
By this relation and the definition of $c_{n}$ we deduce
\begin{equation}
\mathbf{P}(|X|>c_{n})\sim \frac{2-\alpha }{\alpha }\frac{1}{n}\quad \text{as
}n\rightarrow \infty .  \label{tailF'}
\end{equation}%
%
%
%
%
%
%
%
%
%
%
%
%
%
%

\subsection{Some results from the fluctuation theory}

Now we formulate a number of statements concerning the distributions of the
random variables $\tau ^{-},\tau ^{+}$ and $\chi ^{+}$. $\ $\ Recall that a
random variable $\zeta $ is called relatively stable if there exists a
non-random sequence $d_{n}\rightarrow \infty $ as $n\rightarrow \infty $
such that
\begin{equation*}
\frac{1}{d_{n}}\sum_{k=1}^{n}\zeta _{k}\overset{p}{\rightarrow }1\text{ as }%
n\rightarrow \infty ,
\end{equation*}%
where $\zeta _{k}\overset{d}{=}\zeta ,\ k=1,2,...,$ and are independent.

\begin{lemma}
\label{Lrenewal}$\mathrm{(}$see \cite[Theorem~9]{Rog71}$%
\mathrm{)}$ Assume $X\in \mathcal{D}(\alpha ,\beta )$. Then, as~$%
x~\rightarrow~\infty $,%
\begin{equation}
\mathbf{P}\left( \chi ^{+}>x\right) \sim \frac{1}{x^{\alpha \rho }l_{2}(x)}\
\,\text{if \thinspace }\alpha \rho <1,  \label{asymXi}
\end{equation}%
and $\chi ^{+}$ is relatively stable if $\alpha \rho =1.$
\end{lemma}

\begin{lemma}
\label{Renew2} Suppose $X\in \mathcal{D}(\alpha ,\beta )$. Then, as $%
x\rightarrow \infty $,%
\begin{equation}
H(x)\sim \frac{x^{\alpha \rho }l_{2}(x)}{\Gamma (1-\alpha \rho )\Gamma
(1+\alpha \rho )}  \label{RenStand}
\end{equation}%
if $\alpha \rho <1$, and
\begin{equation}
H(x)\sim xl_{3}(x)  \label{RenewRelat}
\end{equation}%
if $\alpha \rho =1,$ where%
\begin{equation*}
l_{3}(x):=\left( \int_{0}^{x}\mathbf{P}\left( \chi ^{+}>y\right) dy\right)
^{-1},\text{ \ }x>0.
\end{equation*}%
In addition, there exists a constant $C>0$ such that, in both cases
\begin{equation}
H(c_{n})\sim Cn\mathbf{P}(\tau ^{-}>n)\quad \text{as }n\rightarrow \infty .
\label{AsH}
\end{equation}
\end{lemma}

\begin{proof}
If $\alpha\rho<1$, then by \cite[Chapter~XIV,~formula~(3.4)]{FE}
\begin{equation*}
H(x)\sim\frac{1}{\Gamma(1-\alpha\rho)\Gamma(1+\alpha\rho)}\frac{1}{\mathbf{P}%
(\chi^{+}>x)}\quad\text{as }x\rightarrow\infty.
\end{equation*}
Hence, recalling (\ref{asymXi}), we obtain (\ref{RenStand}).

If $\alpha\rho=1$, then (\ref{RenewRelat}) follows from Theorem 2 in \cite%
{Rog71}.

Let us demonstrate the validity of \ (\ref{AsH}). We know from \cite{Rog71}
(see also \cite{GOT82}) that $\tau ^{+}\in \mathcal{D}(\rho ,1)$ under the
conditions of the lemma and, in addition, $\chi ^{+}\in \mathcal{D}(\alpha
\rho ,1)$ if $\alpha \rho <1$. This means, in particular, that for sequences
$\{a_{n},n\geq 1\}$ and $\{b_{n},n\geq 1\}$ specified by
\begin{equation}
\mathbf{P}(\tau ^{+}>a_{n})\sim \frac{1}{n}\quad \text{and}\quad \mathbf{P}%
(\chi ^{+}>b_{n})\sim \frac{1}{n}\quad \text{as }n\rightarrow \infty ,
\label{an-asym}
\end{equation}%
and vectors $(\tau _{k}^{+},\chi _{k}^{+}),\,k=1,2,...,$ being independent
copies of $(\tau ^{+},\chi ^{+}),$ we have
\begin{equation}
\frac{1}{a_{n}}\sum_{k=1}^{n}\tau _{k}^{+}\overset{d}{\rightarrow }Y_{\rho
}\quad \text{and}\quad \frac{1}{b_{n}}\sum_{k=1}^{n}\chi _{k}^{+}\overset{d}{%
\rightarrow }Y_{\alpha \rho }\text{ \ \ as \ }n\rightarrow \infty .
\label{**}
\end{equation}%
Moreover, it was established by Doney (see Lemma in \cite{Don85}, p. 358)
that
\begin{equation}
b_{n}\sim Cc_{[a_{n}]}\text{ as }n\rightarrow \infty ,  \label{doney}
\end{equation}%
where $[x]$ stands for the integer part of $x$. Therefore, $c_{n}\sim
Cb_{[a^{-1}(n)]}$, where, with a slight abuse of notation, $a^{-1}(n)$ is
the inverse function to $a_{n}$. Hence, on account of (\ref{an-asym}),
\begin{align}
\mathbf{P}(\chi ^{+}>c_{n})& \sim C_{1}\mathbf{P}(\chi
^{+}>b_{[a^{-1}(n)]})\sim \frac{C_{1}}{a^{-1}(n)}  \notag  \label{***} \\
& \sim C_{2}\mathbf{P}(\tau ^{+}>a_{[a^{-1}(n)]})\sim C_{3}\mathbf{P}(\tau
^{+}>n).
\end{align}

If $\alpha \rho =1$, then, instead of the second equivalence in (\ref%
{an-asym}), one should define $b_{n}$ by
\begin{equation*}
\frac{1}{b_{n}}\int_{0}^{b_{n}}\mathbf{P}(\chi ^{+}>y)dy\sim \frac{1}{n}%
\quad \text{as }n\rightarrow \infty
\end{equation*}%
(see \cite[p.~595]{Rog71}). In this case the second convergence in (\ref{**}%
) transforms to
\begin{equation*}
\frac{1}{b_{n}}\sum_{k=1}^{n}\chi _{k}^{+}\overset{p}{\rightarrow }1\quad
\text{as }n\rightarrow \infty ,
\end{equation*}%
while (\ref{***}) should be changed to
\begin{align}
\frac{1}{c_{n}}\int_{0}^{c_{n}}\mathbf{P}(\chi ^{+}>y)dy& \sim \frac{C_{1}}{%
b_{[a^{-1}(n)]}}\int_{0}^{b_{[a^{-1}(n)]}}\mathbf{P}(\chi ^{+}>y)dy\sim
\frac{C_{1}}{a^{-1}(n)}  \notag \\
& \sim C_{1}\mathbf{P}(\tau ^{+}>a_{[a^{-1}(n)]})\sim C_{2}\mathbf{P}(\tau
^{+}>n).  \label{4*}
\end{align}%
Combining (\ref{***}) and (\ref{4*}) with (\ref{RenStand}) and (\ref%
{RenewRelat}) gives
\begin{equation*}
H(c_{n})\sim C\mathbf{P}(\tau ^{+}>n)\text{ as }n\rightarrow \infty
\end{equation*}%
for all $X\in \mathcal{D}(\alpha ,\beta )$. Using (\ref{integ1}) finishes
the proof of the lemma.
\end{proof}

\begin{lemma}
\label{BehC}If $\mathbf{E}(-S_{\tau ^{-}})<\infty $, then there exists a
positive constant $C_{0}$ such that
\begin{equation}
c_{n}\sim C_{0}\frac{n^{1-\rho }}{l(n)}.  \label{A7}
\end{equation}
\end{lemma}

\begin{proof}
Let $T^{-}:=\min \{k\geq 1:-S_{k}>0\}$ and $\chi ^{-}=-S_{T^{-}}$ be the
first strict ladder height for the random walk $\left\{ -S_{n},n\geq
0\right\} $. Applying (\ref{4*}) to $\left\{ -S_{n},n\geq 0\right\} $, we
have
\begin{equation}
\frac{1}{c_{n}}\int_{0}^{c_{n}}\mathbf{P}(\chi ^{-}>y)dy\sim C\mathbf{P}%
(T^{-}>n).  \label{A8}
\end{equation}%
Obviously, $\mathbf{E}(-S_{\tau ^{-}})<\infty $ yields $\mathbf{E}\chi
^{-}<\infty $. Therefore $\int_{0}^{c_{n}}\mathbf{P}(\chi
^{-}>y)dy\rightarrow \mathbf{E}\chi ^{-}$ as $n\rightarrow \infty $.
Combining this with (\ref{A8}), and recalling that $\mathbf{P}(T^{-}>n)\sim C%
\mathbf{P}(\tau ^{-}>n)$ in view of the equality %
\begin{equation*}
\sum_{n=1}^{\infty }\mathbf{P}(T^{-}>n)z^{n}=\sum_{n=1}^{\infty }\mathbf{P}%
(\tau ^{-}>n)z^{n}\exp \left\{ \sum_{k=1}^{\infty }\frac{z^{k}}{k}\mathbf{P}%
(S_{k}=0)\right\} ,
\end{equation*}%
 asymptotic representation (\ref{integ1}), and the estimate%
\begin{equation*}
\sum_{k=1}^{\infty }\frac{1}{k}\mathbf{P}(S_{k}=0)<\infty ,
\end{equation*}
we obtain
\begin{equation*}
\lim_{n\rightarrow \infty }c_{n}\mathbf{P}(\tau ^{-}>n)=:C_{0}\in (0,\infty
).
\end{equation*}%
On account of (\ref{integ1}) this proves (\ref{A7}).
\end{proof}


\subsection{Upper estimates for local probabilities}

For $x\geq 0$ and $n=0,1,2,...,$ let%
\begin{align*}
\ B_{n}(x)& :=\mathbf{P}\left( S_{n}\in (0,x);\tau ^{-}>n\right) , \\
b_{n}(x)& :=B_{n}(x+1)-B_{n}(x)=\mathbf{P}\left( S_{n}\in \lbrack
x,x+1);\tau ^{-}>n\right) .
\end{align*}

Note that by the duality principle for random walks
\begin{align}
1+\sum_{j=1}^{\infty }B_{j}(x)& =1+\sum_{j=1}^{\infty }\mathbf{P}\left(
S_{j}\in (0,x);\tau ^{-}>j\right)  \notag \\
& =1+\sum_{j=1}^{\infty }\mathbf{P}\left( S_{j}\in
(0,x);S_{j}>S_{0},S_{j}>S_{1},...,S_{j}>S_{j-1}\right)  \notag \\
& =H(x),\quad x>0.  \label{Dual}
\end{align}

\begin{lemma}
\label{Representation} The sequence of functions $\{B_{n}(x),\,n\geq 1\}$
satisfies the recurrence equations
\begin{equation}  \label{rep1}
nB_{n}(x)=\mathbf{P}(S_{n}\in(0,x))+\sum_{k=1}^{n-1}\int_0^x \mathbf{P}%
(S_k\in(0,x-y))dB_{n-k}(y)
\end{equation}
and
\begin{equation}  \label{rep2}
nB_{n}(x)=\mathbf{P}\left( S_{n}\in(0,x)\right) +\sum_{k=1}^{n-1}\int
_{0}^{x}B_{n-k}(x-y)\mathbf{P}\left( S_{k}\in dy\right).
\end{equation}
\end{lemma}

\begin{remark}
The proof of (\ref{rep2}) is contained in Eppel \cite{Epp} (see formula (5)
there). Representation (\ref{rep1}) is not given by Eppel. However, it can
be easily obtained by the same method. Here we demonstrate the mentioned
relations only for the completeness of the presentation.
\end{remark}

\begin{proof}
Let
\begin{equation*}
\mathcal{B}_{n}(t):=\mathbf{E}\left[ e^{itS_{n}};\tau ^{-}>n\right]
=\int_{0}^{\infty }e^{itx}\mathbf{P}\left( {S}_{n}\in dx;\tau ^{-}>n\right)
,\ t\in (-\infty ,\infty ),
\end{equation*}%
be the Fourier transform of the measure $B_{n}$. It is known (see, for
instance, \cite{Spit}, Chapter 4, Section 17) that
\begin{equation*}
1+\sum_{n=1}^{\infty }z^{n}\mathcal{B}_{n}(t)=\exp \left\{
\sum_{k=1}^{\infty }\frac{z^{k}}{k}\mathcal{S}_{k}(t)\right\} ,\ \left\vert
z\right\vert <1,
\end{equation*}%
where $\mathcal{S}_{k}(t):=\mathbf{E}\left[ e^{itS_{k}};S_{k}>0\right] $.
Differentiation with respect to $z$ gives
\begin{equation*}
\sum_{n=1}^{\infty }z^{n-1}n\mathcal{B}_{n}(t)=\left( 1+\sum_{n=1}^{\infty
}z^{n}\mathcal{B}_{n}(t)\right) \sum_{k=1}^{\infty }z^{k-1}\mathcal{S}%
_{k}(t).
\end{equation*}%
Comparing the coefficients of $z^{n-1}$ in the both sides of this equality,
we get
\begin{equation}
n\mathcal{B}_{n}(t)=\mathcal{S}_{n}(t)+\sum_{k=1}^{n-1}\mathcal{B}_{n-k}(t)%
\mathcal{S}_{k}(t).  \label{Eppel}
\end{equation}%
Going back to the distributions, we obtain the desired representations.
\end{proof}

From now on we assume \textit{without loss of generality} that $h=1$ in the
lattice case and, to study the asymptotic behavior of the probabilities of
small deviations when $X$ is $(1,a)$-lattice, introduce a shifted sequence $%
\bar{S}_{n}:=S_{n}-an$ and probabilities $\bar{b}_{n}(x):=\mathbf{P}(\bar{S}%
_{n}=x)=b_{n}(an+x)$. Further, for fixed $x\in \mathbb{Z}$ and $1\leq k\leq
n-1$ set
\begin{equation*}
\mathcal{I}_{x}(k,n):=(-a(n-k),ak+x)\cap \mathbb{Z}.
\end{equation*}

\begin{lemma}
\label{Representation1} The sequence of functions $\{\bar{b}_{n}(x),\,n\geq
1\}$ satisfies the recurrence equation
\begin{equation}  \label{rep}
n\bar{b}_{n}(x)=\mathbf{P}(\bar{S}_{n}=x)+\sum_{k=1}^{n-1}\sum_{y\in
\mathcal{I}_{x}(k,n)}\bar{b}_{k}(x-y)\mathbf{P}(\bar{S}_{n-k}=y).
\end{equation}
\end{lemma}

\begin{remark}
Alili and Doney \cite{AD99} obtained this representation in the case when $X$
is $(h,0)$-lattice.
\end{remark}

\begin{proof}
It follows from (\ref{Eppel}) that
\begin{equation*}
n\bar{b}_{n}(x)=\mathbf{P}(\bar{S}_{n}=x)+\sum_{k=1}^{n-1}\sum\bar{b}%
_{k}(x-y)\mathbf{P}(\bar{S}_{n-k}=y),
\end{equation*}%
where the second sum is taken over all $y\in \mathbb{Z}$ satisfying the
conditions $ak+x-y>0$, $a(n-k)+y>0$. This proves the lemma.
\end{proof}

\begin{lemma}
\label{L0} Assume $X\in \mathcal{D}(\alpha ,\beta )$. Then there exists $C>0
$ such that, for all $y>0$ and all $n\geq 1$,%
\begin{equation}
b_{n}(y)\leq \frac{C}{c_{n}}\frac{l(n)}{n^{1-\rho }}  \label{uupb}
\end{equation}%
and%
\begin{equation}
B_{n}(y)\leq \frac{C\left( y+1\right) }{c_{n}}\frac{l(n)}{n^{1-\rho }}.
\label{UpB}
\end{equation}
\end{lemma}

\begin{proof}
For $n=1$ the statement of the lemma is obvious. Let $\left\{ S_{n}^{\ast
},n\geq 0\right\} $ be a random walk distributed as $\left\{ S_{n},n\geq
0\right\} $ and independent of it. One can easily check that for each $n\geq
2$,%
\begin{align}
b_{n}(y)& =\mathbf{P}\left( y\leq S_{n}<y+1;\tau ^{-}>n\right)  \notag \\
& =\int_{0}^{\infty }\mathbf{P}\Bigl(y-S_{[n/2]}\leq
S_{n}-S_{[n/2]}<y+1-S_{[n/2]};S_{[n/2]}\in dz;\tau ^{-}>n\Bigr)  \notag \\
& \leq \int_{0}^{\infty }\mathbf{P}\Bigl(y-z\leq S_{n-[n/2]}^{\ast
}<y+1-z;S_{[n/2]}\in dz;\tau ^{-}>[n/2]\Bigr)  \notag \\
& \leq \mathbf{P}\Bigl(\tau ^{-}>[n/2]\Bigr)\sup_{z}\mathbf{P}\Bigl(z\leq
S_{n-[n/2]}^{\ast }<z+1\Bigr).  \label{Sm1}
\end{align}%
Since the density of any $\alpha $-stable law is bounded, it follows from
the Gnedenko and Stone local limit theorems that there exists a constant $%
C>0 $ such that for all $n\geq 1$ and all $z\geq 0,$
\begin{equation}
\mathbf{P}\left( S_{n}\in \lbrack z,z+\Delta )\right) \leq \frac{C\Delta }{%
c_{n}}.  \label{EstS1}
\end{equation}%
Hence it follows, in particular, that, for any $z>0$,
\begin{equation}
\mathbf{P}\left( S_{n}\in \lbrack 0,z)\right) \leq \frac{C(z+1)}{c_{n}}.
\label{EstS2}
\end{equation}%
Substituting (\ref{EstS1}) into (\ref{Sm1}), and recalling (\ref{asyma}) and
properties of regularly varying functions, we get (\ref{uupb}). Estimate (%
\ref{UpB}) follows from (\ref{uupb}) by summation.
\end{proof}

\begin{lemma}
\label{Concentration} If $X\in \mathcal{D}(\alpha ,\beta )$ then there
exists a constant $C\in (0,\infty )$ such that
\begin{equation}
b_{n}(x)\leq C\frac{H(x)}{nc_{n}}  \label{P1}
\end{equation}%
and
\begin{equation}
B_{n}(x)\leq C\frac{xH(x)}{nc_{n}}  \label{P2}
\end{equation}%
for all $n\geq 1$ and all $x\in (0,c_{n}]$.
\end{lemma}

\begin{remark}
Comparing (\ref{P1}) and (\ref{SD}) (to be proved later), we see that, in
the domain of small deviations, the estimates given by the lemma are optimal
up to a constant factor.
\end{remark}

\begin{proof}
By (\ref{rep2}) we get%
\begin{align}
nb_{n}(x)& =\mathbf{P}\left( S_{n}\in \lbrack x,x+1)\right)
+\sum_{k=1}^{n-1}\int_{0}^{x}b_{n-k}(x-y)\mathbf{P}\left( S_{k}\in dy\right)
\notag \\
& \qquad +\sum_{k=1}^{n-1}\int_{x}^{x+1}B_{n-k}(x+1-y)\mathbf{P}\left(
S_{k}\in dy\right) .  \label{Smallb}
\end{align}

Using (\ref{uupb}), (\ref{EstS2}) and properties of slowly varying
functions, we deduce
\begin{align}
\sum_{k=1}^{[n/2]}\int_{0}^{x}b_{n-k}(x-y)\mathbf{P}\left( S_{k}\in
dy\right) & \leq C\sum_{k=1}^{[n/2]}\frac{l(n-k)}{c_{n-k}\left( n-k\right)
^{1-\rho }}\mathbf{P}\left( S_{k}\in \lbrack 0,x)\right)  \notag \\
& \leq \frac{C_{1}}{c_{n}}\frac{l(n)}{n^{1-\rho }}\sum_{k=1}^{[n/2]}\mathbf{P%
}\left( S_{k}\in \lbrack 0,x)\right) .  \label{EstSmall}
\end{align}

On the other hand, in view of (\ref{EstS1}) and monotonicity of $B_{k}(x)$
in $x$ we conclude (assuming that $x$ is integer without loss of generality
and letting $B_{k}(-1)=0$ and $H(-1)=0$) that
\begin{align*}
& \sum_{k=[n/2]+1}^{n}\int_{0}^{x}b_{n-k}(x-y)\mathbf{P}\left( S_{k}\in
dy\right) \\
& \leq\sum_{k=[n/2]+1}^{n}\sum_{j=0}^{x}\left(
B_{n-k}(x-j+1)-B_{n-k}(x-j-1)\right) \mathbf{P}\left( S_{k}\in[j,j+1)\right)
\\
& \leq\sum_{k=[n/2]+1}^{n}\sum_{j=0}^{x}\left(
B_{n-k}(x-j+1)-B_{n-k}(x-j-1)\right) \frac{C}{c_{k}} \\
& \leq\frac{C}{c_{n}}\sum_{j=0}^{x}\sum_{k=0}^{\infty}\left(
B_{k}(x-j+1)-B_{k}(x-j-1)\right) \\
& =\frac{C}{c_{n}}\sum_{j=0}^{x}\left( H(x-j+1)-H(x-j-1)\right) \\
& \leq\frac{C}{c_{n}}\left( H(x)+H(x+1)\right) \leq\frac{2C}{c_{n}}H(x+1),
\end{align*}
where for the intermediate equality we have used (\ref{Dual}). This gives
\begin{equation}
\sum_{k=[n/2]+1}^{n}\int_{0}^{x}b_{n-k}(x-y)\mathbf{P}\left( S_{k}\in
dy\right) \leq\frac{C}{c_{n}}H(x+1).  \label{EstBig}
\end{equation}

Since $x\mapsto B_{n}(x)$ increases for every $n$,
\begin{equation}
\sum_{k=1}^{n-1}\int_{x}^{x+1}B_{n-k}(x+1-y)\mathbf{P}\left( S_{k}\in
dy\right) \leq \sum_{k=1}^{n-1}B_{n-k}(1)\mathbf{P}(S_{k}\in \lbrack x,x+1)).
\label{added1}
\end{equation}%
Further, in view of (\ref{UpB}) and (\ref{EstS1}) we have
\begin{equation}
\sum_{k=1}^{[n/2]}B_{n-k}(1)\mathbf{P}(S_{k}\in \lbrack x,x+1))\leq \frac{%
C_{1}}{c_{n}}\frac{l(n)}{n^{1-\rho }}\sum_{k=1}^{[n/2]}\mathbf{P}(S_{k}\in
\lbrack x,x+1)).  \label{added2}
\end{equation}%
Applying (\ref{EstS1}) once again yields
\begin{equation}
\sum_{k=[n/2]+1}^{n-1}B_{n-k}(1)\mathbf{P}(S_{k}\in \lbrack x,x+1))\leq
\frac{C}{c_{n}}\sum_{k=[n/2]+1}^{n-1}B_{n-k}(1)\leq \frac{C}{c_{n}}H(1).
\label{added3}
\end{equation}%
Combining (\ref{Smallb})-(\ref{added3}) and using the monotonicity of $H(x)$%
, we obtain the estimate
\begin{equation*}
nb_{n}(x)\leq \frac{C}{c_{n}}\Bigl(H(x+1)+\frac{l(n)}{n^{1-\rho }}%
\sum_{k=1}^{[n/2]}\mathbf{P}\left( S_{k}\in \lbrack 0,x+1)\right) \Bigr).
\end{equation*}%
Therefore, to complete the proof of (\ref{P1}) it remains to show that
\begin{equation}
\frac{l(n)}{n^{1-\rho }}\sum_{k=1}^{[n/2]}\mathbf{P}\left( S_{k}\in \mathbf{[%
}0,x+1)\right) \leq CH(x+1).  \label{LastStep}
\end{equation}

This will be done separately for the cases $\alpha \in (1,2],\alpha \in
(0,1),$ and $\alpha =1$.

Consider first the case $\alpha \in (1,2]$. It follows from (\ref{EstS2})
that
\begin{equation}
\sum_{k=1}^{[n/2]}\mathbf{P}(0\leq S_{k}<x+1)\leq C(x+1)\sum_{k=1}^{n}\frac{1%
}{c_{k}}\leq C(x+1)\frac{n}{c_{n}},  \label{Summa1}
\end{equation}%
where at the last step we have used the relation
\begin{equation}
\sum_{k=1}^{n}\frac{1}{c_{k}}\sim \frac{\alpha }{\alpha -1}\frac{n}{c_{n}}%
\quad \text{as }n\rightarrow \infty .  \label{57}
\end{equation}%
By Lemma~\ref{Renew2} and properties of regularly varying functions we
conclude that there exists a non-decreasing function $\phi (u)$ such that $%
u/H(u)\sim \phi (u)$ as $u\rightarrow \infty $. Therefore, for any $%
\varepsilon \in (0,1/2)$ there exists a $u_{0}=u_{0}(\varepsilon )$ such
that, for all $u\geq u_{0}$,%
\begin{equation*}
(1-\varepsilon )\phi (u)\leq \frac{u}{H(u)}\leq (1+\varepsilon )\phi (u).
\end{equation*}%
From this estimate it is not difficult to conclude that there exists a
constant $C$ such that, for all $n\geq 1$ and all $x\in (0,c_{n}]$,
\begin{equation*}
\frac{x}{H(x)}\leq C\frac{c_{n}}{H(c_{n})}.
\end{equation*}%
Hence we see that the right-hand side of (\ref{Summa1}) is bounded from
above by
\begin{equation*}
C\frac{nH(x+1)}{H(c_{n})}.
\end{equation*}%
Recalling that $H(x)$ is regularly varying as $x\rightarrow \infty $, and
applying (\ref{AsH}) and (\ref{integ1}), we finally arrive at the inequality
\begin{equation*}
\sum_{k=1}^{[n/2]}\mathbf{P}(0\leq S_{k}<x+1)\leq CH(x+1)\frac{n^{1-\rho }}{%
l(n)}.
\end{equation*}%
This justifies (\ref{LastStep}) for $\alpha \in (1,2]$.

Now we turn to the case $\alpha \in (0,1)$. Letting $N_{x}:=\max \{k\geq
1:c_{k}\leq x+1\}$ and applying (\ref{EstS1}), we get
\begin{align}
\sum_{k=1}^{[n/2]}\mathbf{P}(0\leq S_{k}<x+1)& \leq
N_{x}+C(x+1)\sum_{k=N_{x}+1}^{n}\frac{1}{c_{k}}  \label{new1} \\
& \leq N_{x}+C(x+1)\frac{N_{x}}{c_{N_{x}+1}}.  \notag
\end{align}%
Here we have used the asymptotic representation
\begin{equation*}
\sum_{k=n+1}^{\infty }\frac{1}{c_{k}}\sim \frac{\alpha }{1-\alpha }\frac{n}{%
c_{n+1}}\quad \text{as }n\rightarrow \infty .
\end{equation*}%
If $\alpha =1$, then, in view of (\ref{asyma}),
\begin{equation*}
\sum_{k=N_{x}+1}^{n}\frac{1}{c_{k}}=\frac{N_{x}+1}{c_{N_{x}+1}}%
\sum_{k=N_{x}+1}^{n}\frac{l_{1}(N_{x}+1)}{l_{1}(k)}\frac{1}{k}.
\end{equation*}%
From the Karamata representation for slowly varying functions (see \cite%
{Sen76}, Theorem 1.2) we conclude that for every slowly varying function $%
l^{\ast }(x)$ and every $\gamma >0$ there exists a constant $C=C\left(
\gamma \right) $ such that
\begin{equation}
\frac{l^{\ast }(x)}{l^{\ast }(y)}\leq C\max \left\{ \left( \frac{x}{y}%
\right) ^{\gamma },\left( \frac{x}{y}\right) ^{-\gamma }\right\} \quad \text{%
for all }x,y>0.  \label{prop}
\end{equation}%
Applying this inequality to $l_{1}(x)$, we obtain
\begin{equation*}
\sum_{k=N_{x}+1}^{n}\frac{1}{c_{k}}\leq C\frac{N_{x}+1}{c_{N_{x}+1}}\Bigl(%
\frac{n}{N_{x}+1}\Bigr)^{\gamma }\log (\frac{n}{N_{x}+1}\Bigr).
\end{equation*}%
Combining this bound with (\ref{new1}), and using the inequality $%
c_{N_{x}+1}\geq x+1$, we conclude that
\begin{equation}
\sum_{k=1}^{[n/2]}\mathbf{P}(0\leq S_{k}<x+1)\leq C_{1}N_{x}\Bigl(\frac{n}{%
N_{x}}\Bigr)^{2\gamma }  \label{new2}
\end{equation}%
for all $\alpha \in (0,1]$. Consequently,
\begin{equation*}
\frac{l(n)}{n^{1-\rho }}\sum_{k=1}^{[n/2]}\mathbf{P}(0\leq S_{k}<x+1)\leq
C_{1}H(x+1)\Bigl(\frac{n}{N_{x}}\Bigr)^{2\gamma }\frac{l(n)N_{x}}{n^{1-\rho
}H(x+1)}.
\end{equation*}

The definition of $N_{x},$ (\ref{AsH}), and (\ref{integ1}) imply
\begin{equation*}
H(x+1)\geq H(c_{N_{x}})\geq Cl(N_{x})N_{x}^{\rho }.
\end{equation*}%
Therefore,
\begin{equation*}
\frac{l(n)}{n^{1-\rho }}\sum_{k=1}^{[n/2]}\mathbf{P}(0\leq S_{k}<x+1)\leq
C_{1}H(x+1)\Bigl(\frac{N_{x}}{n}\Bigr)^{1-\rho -2\gamma }\frac{l(n)}{l(N_{x})%
}.
\end{equation*}%
Applying (\ref{prop}) to $l(x)$ and choosing $\gamma =(1-\rho )/4$, we
finally arrive at the inequality
\begin{equation}
\frac{l(n)}{n^{1-\rho }}\sum_{k=1}^{[n/2]}\mathbf{P}(0\leq S_{k}<x+1)\leq
CH(x+1)\Bigl(\frac{N_{x}}{n}\Bigr)^{(1-\rho )/4}\leq CH(x+1).  \label{new3}
\end{equation}%
establishing (\ref{LastStep}) for $\alpha \in (0,1]$. Thus, (\ref{LastStep})
is justified for all $X\in \mathcal{D}(\alpha ,\beta )$, and, consequently, (%
\ref{P1}) is proved.

The second statement of the lemma follows by summation.
\end{proof}

Later on we need the following refined version of Lemma \ref{Concentration}:

\begin{corollary}
\label{Cmin}Suppose $X\in \mathcal{D}(\alpha ,\beta ).$ Then there exists a
constant $C\in \left( 0,\infty \right) $ such that, for all $n\geq 1$,%
\begin{equation}
b_{n}(x)\leq C\frac{H(\min (c_{n},x))}{nc_{n}}  \label{P11}
\end{equation}%
and
\begin{equation}
B_{n}(x)\leq C\frac{\min (c_{n},x)H(\min (c_{n},x))}{nc_{n}}.  \label{P22}
\end{equation}
\end{corollary}

\begin{proof}
The desired estimates follow from (\ref{uupb}), (\ref{UpB}) and Lemma \ref%
{Concentration}.
\end{proof}

\begin{lemma}
\label{LDensity} There exists a constant $C\in \left( 0,\infty \right) $
such that, for all $z\in \lbrack 0,\infty ),$%
\begin{equation*}
\lim \sup_{\varepsilon \downarrow 0}\varepsilon ^{-1}\mathbf{P}\left(
M_{\alpha ,\beta }\in \lbrack z,z+\varepsilon )\right) \leq C\min
\{1,z^{\alpha \rho }\}.
\end{equation*}%
In particular,
\begin{equation*}
\lim_{z\downarrow 0}\lim \sup_{\varepsilon \downarrow 0}\varepsilon ^{-1}%
\mathbf{P}\left( M_{\alpha ,\beta }\in \lbrack z,z+\varepsilon )\right) =0.
\end{equation*}
\end{lemma}

\begin{proof}
For all $z\geq 0$ and all $\varepsilon >0$ we have
\begin{equation*}
\mathbf{P}\left( M_{\alpha ,\beta }\in \lbrack z,z+\varepsilon )\right) \leq
\lim \sup_{n\rightarrow \infty }\mathbf{P}\left( S_{n}\in \lbrack
c_{n}z,c_{n}(z+\varepsilon ))|\tau ^{-}>n\right) .
\end{equation*}%
Applying (\ref{P11}) gives
\begin{equation*}
\mathbf{P}\left( S_{n}\in \lbrack c_{n}z,c_{n}(z+\varepsilon ))|\tau
^{-}>n\right) \leq C\frac{H(\min \left( c_{n},(z+\varepsilon )c_{n}\right) )%
}{nc_{n}\mathbf{P}(\tau ^{-}>n)}\varepsilon c_{n}.
\end{equation*}%
Recalling that $H(x)$ is regularly varying with index $\alpha \rho $ by
Lemma \ref{Renew2} and taking into account (\ref{AsH}), we get
\begin{align*}
\mathbf{P}\left( S_{n}\in \lbrack c_{n}z,c_{n}(z+\varepsilon ))|\tau
^{-}>n\right) & \leq C\varepsilon \min \{1,(z+\varepsilon )^{\alpha \rho }\}%
\frac{H(c_{n})}{n\mathbf{P}(\tau ^{-}>n)} \\
& \leq C\varepsilon \min \{1,(z+\varepsilon )^{\alpha \rho }\}.
\end{align*}%
Consequently,
\begin{equation}
\mathbf{P}(M_{\alpha ,\beta }\in \lbrack z,z+\varepsilon ))\leq C\varepsilon
\min \{1,(z+\varepsilon )^{\alpha \rho }\}.  \label{Majorante}
\end{equation}%
This inequality shows that there exists a constant $C\in (0,\infty )$ such
that
\begin{equation*}
\lim \sup_{\varepsilon \downarrow 0}\varepsilon ^{-1}\mathbf{P}\left(
M_{\alpha ,\beta }\in \lbrack z,z+\varepsilon )\right) \leq C\min
\{1,z^{\alpha \rho }\}\text{ for all \ }z\geq 0
\end{equation*}%
as desired.
\end{proof}


\section{Probabilities of Normal deviations: Proofs of Theorems~\protect\ref%
{NormalDev'} and \protect\ref{NormalDev}}

The first part of the proof to follow is one and the same for non-lattice \
(Theorem \ref{NormalDev'})\ and lattice (Theorem \ref{NormalDev})
cases.

It follows from (\ref{rep1}) that
\begin{align}
nb_{n}(x)=\mathbf{P}(S_{n}\in \lbrack x,x+1))& +\sum_{k=1}^{n-1}\int_{0}^{x}%
\mathbf{P}(S_{k}\in \lbrack x-y,x-y+1))dB_{n-k}(y)  \notag \\
& +\sum_{k=1}^{n-1}\int_{x}^{x+1}\mathbf{P}(S_{k}\in (0,x-y+1))dB_{n-k}(y)
\notag \\
& =:R_{\varepsilon }^{(1)}(x)+R_{\varepsilon }^{(2)}(x)+R_{\varepsilon
}^{(3)}(x)+R^{(0)}(x),  \label{c1}
\end{align}%
where, for any fix $\varepsilon \in (0,1/2)$ and with a slight abuse of
notation
\begin{equation*}
R_{\varepsilon }^{(1)}(x):=\sum_{k=1}^{[\varepsilon n]}\int_{0}^{x}\mathbf{P}%
(S_{k}\in \lbrack x-y,x-y+1))dB_{n-k}(y),
\end{equation*}%
\begin{equation*}
R_{\varepsilon }^{(2)}(x):=\sum_{k=[\varepsilon n]+1}^{[(1-\varepsilon
)n]}\int_{0}^{x}\mathbf{P}(S_{k}\in \lbrack x-y,x-y+1))dB_{n-k}(y),
\end{equation*}%
\begin{equation*}
{R}_{\varepsilon }^{(3)}(x):=\mathbf{P}({S}_{n}\in \lbrack
x,x+1))+\sum_{k=[(1-\varepsilon )n]+1}^{n-1}\int_{0}^{x}\mathbf{P}(S_{k}\in
\lbrack x-y,x-y+1))dB_{n-k}(y),
\end{equation*}%
and%
\begin{equation*}
R^{(0)}(x):=\sum_{k=1}^{n-1}\int_{x}^{x+1}\mathbf{P}(S_{k}\in
(0,x-y+1))dB_{n-k}(y).
\end{equation*}%
First observe that
\begin{equation*}
R^{(0)}(x)\leq \sum_{k=1}^{n-1}\mathbf{P}(S_{k}\in (0,1))b_{n-k}(x).
\end{equation*}%
Applying Corollary \ref{Cmin} we may simplify the estimate above to
\begin{align}
R^{(0)}(x)& \leq C\sum_{k=1}^{n-1}\mathbf{P}(S_{k}\in (0,1))\frac{%
H_{n-k}(c_{n-k})}{(n-k)c_{n-k}}  \notag \\
& \leq C\frac{H(c_{n})}{nc_{n}}\sum_{k=1}^{[n/2]}\mathbf{P}(S_{k}\in (0,1))+%
\frac{C}{c_{n}}\sum_{k=1}^{[n/2]}\frac{H(c_{k})}{kc_{k}},  \label{c2}
\end{align}%
where at the last step we have used the properties of $c_{n}$ and inequality
(\ref{EstS1}).

Since $H(x)\leq Cx$, we have
\begin{equation*}
\sum_{k=1}^{[n/2]}\frac{H(c_{k})}{kc_{k}}\leq C\sum_{k=1}^{[n/2]}\frac{1}{k}%
\leq C\log n.
\end{equation*}%
Further, by (\ref{Summa1}) and (\ref{new2}) with $x=0$ we know that
\begin{equation*}
\sum_{k=1}^{[n/2]}\mathbf{P}(S_{k}\in (0,1))\leq C\left( \frac{n}{c_{n}}%
\mathrm{I}(\alpha \in (1,2])+n^{\gamma }\mathrm{I}(\alpha \in (0,1])\right)
\leq C\left( \frac{n}{c_{n}}+n^{\gamma }\right) .
\end{equation*}%
Substituting these estimates into (\ref{c2}) leads to the inequalities
\begin{equation*}
R^{(0)}(x)\leq \frac{C}{c_{n}}\left( \frac{H(c_{n})}{c_{n}}+\frac{H(c_{n})}{%
n^{1-\gamma }}+\log n\right) .
\end{equation*}%
By these relations, recalling that $\mathbf{P}(\tau ^{-}>n)$ is
regularly varying with index\break $\rho-1>-1$ (see
(\ref{integ1})) and using (\ref{AsH}), we obtain
\begin{equation}
\limsup_{n\rightarrow \infty }\frac{c_{n}}{n\mathbf{P}(\tau ^{-}>n)}%
R^{(0)}(x)=0.  \label{c3}
\end{equation}

Now we evaluate the remaining terms in (\ref{c1}).

In view of (\ref{EstS1})
\begin{equation*}
R_{\varepsilon }^{(3)}(x)\leq \frac{C}{c_{n}}\left(
1+\sum_{k=1}^{[\varepsilon n]}B_{k}(x)\right) \leq \frac{C}{c_{n}}%
\sum_{k=0}^{[\varepsilon n]}\mathbf{P}(\tau ^{-}>k)
\end{equation*}
for all $x>0$. \ Further, by (\ref{integ1})
\begin{equation*}
\sum_{k=0}^{[\varepsilon n]}\mathbf{P}(\tau ^{-}>k)\sim \rho
^{-1}\varepsilon ^{\rho }n\mathbf{P}(\tau ^{-}>n)\quad \text{as }%
n\rightarrow \infty .
\end{equation*}%
As a result we obtain
\begin{equation}
\limsup_{n\rightarrow \infty }\frac{c_{n}}{n\mathbf{P}(\tau ^{-}>n)}%
\sup_{x>0}R_{\varepsilon }^{(3)}(x)\leq C\varepsilon ^{\rho }.  \label{n1}
\end{equation}

Using the inequalities
\begin{equation}
\int_{j}^{j+1}\mathbf{P}(S_{k}\in \lbrack x-y,x-y+1))dB_{n-k}(y)\leq \mathbf{%
P}(S_{k}\in \lbrack x-j-1,x-j+1))b_{n-k}(j)  \label{n5}
\end{equation}%
and
\begin{equation}
\int_{\lbrack x]}^{x}\mathbf{P}(S_{k}\in \lbrack x-y,x-y+1))dB_{n-k}(y)\leq
\mathbf{P}(S_{k}\in \lbrack 0,2))b_{n-k}([x]),  \label{n6}
\end{equation}%
and applying Corollary~\ref{Cmin}, we get
\begin{equation*}
R_{\varepsilon }^{(1)}(x)\leq C\frac{H(c_{n})}{nc_{n}}\sum_{k=1}^{[%
\varepsilon n]}\mathbf{P}(0<S_{k}<x)\leq \varepsilon C\frac{H(c_{n})}{c_{n}}.
\end{equation*}%
From this estimate and (\ref{AsH}) we deduce
\begin{equation}
\limsup_{n\rightarrow \infty }\frac{c_{n}}{n\mathbf{P}(\tau ^{-}>n)}%
R_{\varepsilon }^{(1)}(x)\leq C\varepsilon .  \label{n2}
\end{equation}

Evaluating $R_{\varepsilon }^{(2)}(x)$\textbf{\ }we have to distinguish the
non-lattice \textbf{(Theorem \ref{NormalDev'}) and lattice (Theorem \ref%
{NormalDev})} cases. Detailed estimates are given for the non-lattice case
only. To deduce the respective estimates for the lattice case one should use
the Gnedenko local limit theorem instead of the Stone local limit theorem.

Thus, in the non-lattice case we combine the Stone local limit theorem with
the first equality in (\ref{meander1}) and obtain, uniformly in $x>0$, as $%
n\rightarrow \infty $,
\begin{align*}
R_{\varepsilon }^{(2)}(x)& =\sum_{k=[\varepsilon n]+1}^{[(1-\varepsilon )n]}%
\frac{1}{c_{n-k}}\int_{0}^{x}g_{\alpha ,\beta }\left( \frac{x-y}{c_{n-k}}%
\right) dB_{n-k}(y)+o\left( \frac{1}{c_{n\varepsilon }}%
\sum_{k=1}^{n}B_{k}(x)\right) \\
& =\sum_{k=[\varepsilon n]+1}^{[(1-\varepsilon )n]}\frac{\mathbf{P}(\tau
^{-}>k)}{c_{n-k}}\int_{0}^{x/c_{k}}g_{\alpha ,\beta }\left( \frac{x-c_{k}u}{%
c_{n-k}}\right) \mathbf{P}(M_{\alpha ,\beta }\in du) \\
& \hspace{2cm}+o\left( \frac{1}{c_{n\varepsilon }}\sum_{k=1}^{n}B_{k}(x)+%
\sum_{k=1}^{n-1}\frac{\mathbf{P}(\tau ^{-}>k)}{c_{n\varepsilon }}\right) .
\end{align*}%
According to (\ref{integ1})
\begin{equation*}
\sum_{k=1}^{n}B_{k}(x)\leq \sum_{k=1}^{n}\mathbf{P}(\tau ^{-}>k)\leq Cn%
\mathbf{P}(\tau ^{-}>n).
\end{equation*}%
Hence it follows that
\begin{align*}
R_{\varepsilon }^{(2)}(x)& =\sum_{k=[\varepsilon n]+1}^{[(1-\varepsilon )n]}%
\frac{\mathbf{P}(\tau ^{-}>k)}{c_{n-k}}\int_{0}^{x/c_{k}}g_{\alpha ,\beta
}\left( \frac{x-c_{k}u}{c_{n-k}}\right) \mathbf{P}(M_{\alpha ,\beta }\in du)
\\
& +o\left( \frac{n\mathbf{P}(\tau ^{-}>n)}{c_{n\varepsilon }}\right) .
\end{align*}%
Since $c_{k}$ and $\mathbf{P}(\tau ^{-}>k)$ are regularly varying and $%
g_{\alpha ,\beta }(x)$ is uniformly continuous in $(-\infty ,\infty )$, we
let, for brevity, $v=x/c_{n}$ and continue the previous estimates for $%
R_{\varepsilon }^{(2)}(x)$ with
\begin{align*}
& =\frac{\mathbf{P}(\tau ^{-}>n)}{c_{n}}\sum_{k=[\varepsilon
n]+1}^{[(1-\varepsilon )n]}\frac{(k/n)^{\rho -1}}{(1-k/n)^{1/\alpha }}%
\int_{0}^{v/(k/n)^{1/\alpha }}\hspace{-0.5cm}g_{\alpha ,\beta }\left( \frac{%
v-(k/n)^{1/\alpha }u}{(1-k/n)^{1/\alpha }}\right) \mathbf{P}(M_{\alpha
,\beta }\in du) \\
& \hspace{5cm}+o\left( \frac{n\mathbf{P}(\tau ^{-}>n)}{c_{n\varepsilon }}%
\right) \\
& =\frac{n\mathbf{P}(\tau ^{-}>n)}{c_{n}}f(\varepsilon ,1-\varepsilon
;v)+o\left( \frac{n\mathbf{P}(\tau ^{-}>n)}{c_{n\varepsilon }}\right)
\end{align*}%
where, for $0\leq w_{1}\leq w_{2}\leq 1$,
\begin{equation}
f(w_{1},w_{2};v):=\int_{w_{1}}^{w_{2}}\frac{t^{\rho -1}dt}{(1-t)^{1/\alpha }}%
\int_{0}^{v/t^{1/\alpha }}\hspace{-0.5cm}g_{\alpha ,\beta }\left( \frac{%
v-t^{1/\alpha }u}{(1-t)^{1/\alpha }}\right) \mathbf{P}(M_{\alpha ,\beta }\in
du).  \label{DenF}
\end{equation}%
Observe that by boundness of $g_{\alpha ,\beta }\left( y\right) $
\begin{equation*}
f(0,\varepsilon ;v)\leq C\int_{0}^{\varepsilon }t^{\rho -1}dt\leq
C\varepsilon ^{\rho }.
\end{equation*}%
Further, it follows from (\ref{Majorante}) that $\int \phi (u)\mathbf{P}%
(M_{\alpha ,\beta }\in du)\leq C\int \phi (u)du$ for every non-negative
integrable function $\phi $. Therefore,
\begin{align*}
& f(1-\varepsilon ,1;v) \\
& \leq C\int_{1-\varepsilon }^{1}\frac{t^{\rho -1}dt}{(1-t)^{1/\alpha }}%
\int_{0}^{v/t^{1/\alpha }}\hspace{-0.5cm}g_{\alpha ,\beta }\left( \frac{%
v-t^{1/\alpha }u}{(1-t)^{1/\alpha }}\right) du=\left( z=\frac{v-t^{1/\alpha
}u}{(1-t)^{1/\alpha }}\right) \\
& =C\int_{1-\varepsilon }^{1}t^{\rho -1-1/\alpha
}dt\int_{0}^{v/(1-t)^{1/\alpha }}\hspace{-0.5cm}g_{\alpha ,\beta }\left(
z\right) dz\leq C\varepsilon .
\end{align*}%
As a result we have
\begin{equation}
\limsup_{n\rightarrow \infty }\sup_{x>0}\left\vert \frac{c_{n}}{n\mathbf{P}%
(\tau ^{-}>n)}R_{\varepsilon }^{(2)}(x)-f(0,1;x/c_{n})\right\vert \leq
C\varepsilon ^{\rho }.  \label{n3}
\end{equation}%
Combining (\ref{c3}) -- (\ref{n3}) with representation (\ref{c1}) leads to
\begin{equation}
\limsup_{n\rightarrow \infty }\sup_{x>0}\left\vert \frac{c_{n}}{\mathbf{P}%
(\tau ^{-}>n)}b_{n}(x)-f(0,1;x/c_{n})\right\vert \leq C\varepsilon ^{\rho }.
\label{n4}
\end{equation}%
Since $\varepsilon >0$ is arbitrary, it follows that, as $n\rightarrow
\infty $%
\begin{equation}
\frac{c_{n}}{\mathbf{P}(\tau ^{-}>n)}b_{n}(x)-f(0,1;x/c_{n})\rightarrow 0
\label{Fconv}
\end{equation}%
uniformly in $x>0$. Recalling (\ref{meander1}), we deduce by integration of (%
\ref{Fconv}) and evident transformations that
\begin{equation}
\int_{u_{1}}^{u_{2}}f(0,1;z)dz=\mathbf{P}(M_{\alpha ,\beta }\in \lbrack
u_{1},u_{2}])  \label{deniq}
\end{equation}%
for all $0<u_{1}<u_{2}<\infty $. This means, in particular, that the
distribution of $M_{\alpha ,\beta }$ is absolutely continuous. Furthermore,
it is not difficult to see that $z\mapsto f(0,1;z)$ is a continuous mapping.
Hence, in view of (\ref{deniq}), we may consider $f(0,1;z)$ as a continuous
version of the density of the distribution of $M_{\alpha ,\beta }$ and let $%
p_{\alpha ,\beta }(z):=$ $f(0,1;z).$ This and (\ref{Fconv}) \ imply the
statement of Theorem \ref{NormalDev'} for $\Delta =1$. To establish the
desired result for arbitrary $\Delta >0$ it suffices to consider the random
walk $S_{n}/\Delta $ and to observe that
\begin{equation*}
c_{n}^{\Delta }:=\inf \left\{ u\geq 0:\frac{1}{u^{2}}\int_{-u}^{u}x^{2}P%
\left( \frac{X}{\Delta }\in dx\right) \right\} =c_{n}/\Delta .
\end{equation*}

Note that (\ref{DenF}) gives an interesting representation for $p_{\alpha
,\beta }(v):$%
\begin{equation}
p_{\alpha ,\beta }(z)=\int_{0}^{1}\frac{t^{\rho -1}dt}{(1-t)^{1/\alpha }}%
\int_{0}^{z/t^{1/\alpha }}\hspace{-0.5cm}g_{\alpha ,\beta }\left( \frac{%
z-t^{1/\alpha }u}{(1-t)^{1/\alpha }}\right) p_{\alpha ,\beta }(u)du.
\label{C0}
\end{equation}%
Besides, it follows from Lemma \ref{LDensity} that
\begin{equation*}
p_{\alpha ,\beta }(z)\leq C\min \{1,z^{\alpha \rho }\}
\end{equation*}%
and, that is not surprising,
\begin{equation}
\lim_{z\downarrow 0}p_{\alpha ,\beta }(z)=0.  \label{75}
\end{equation}%
In section \ref{Sden} we refine these statements.


\section{Probabilities of Small deviations}

\subsection{Lattice case: Proof of Theorem~\protect\ref{SmallDev}}

Recall that the span $h=1$ according to our agreement. Fix any $\varepsilon
\in (0,1/2)$ and, using Lemma~\ref{Representation1}, write
\begin{equation}
n\bar{b}_{n}(x)=R_{\varepsilon n}(x)+\bar{R}_{\varepsilon n}(x),  \label{S1}
\end{equation}%
where
\begin{equation*}
R_{\varepsilon n}(x):=\mathbf{P}(\bar{S}_{n}=x)+\sum_{k=1}^{[\varepsilon
n]}\sum_{y\in \mathcal{I}_{x}(k,n)}\bar{b}_{k}(x-y)\mathbf{P}(\bar{S}%
_{n-k}=y)
\end{equation*}%
and%
\begin{equation*}
\bar{R}_{\varepsilon n}(x):=\sum_{k=[\varepsilon n]+1}^{n-1}\sum_{y\in
\mathcal{I}_{x}(k,n)}\bar{b}_{k}(x-y)\mathbf{P}(\bar{S}_{n-k}=y).
\end{equation*}%
In view of Lemma~\ref{Concentration},
\begin{align*}
\bar{R}_{\varepsilon n}(x)& \leq C\sum_{k=[\varepsilon n]+1}^{n-1}\sum_{y\in
\mathcal{I}_{x}(k,n)}\frac{H(ak+x-y)}{kc_{k}}\mathbf{P}(\bar{S}_{n-k}=y) \\
& \leq C(\varepsilon )\frac{H(an+x)}{nc_{n}}\sum_{k=1}^{n-\left[ \varepsilon
n\right] }\mathbf{P}(0\leq S_{k}<an+x).
\end{align*}%
Introduce the set%
\begin{equation*}
\mathcal{G}_{n}:=(-an,-an+\delta _{n}c_{n}]\cap \mathbb{Z}.
\end{equation*}%
Taking into account \ estimate (\ref{Summa1}) (with $[n/2]$ replaced by $%
n-[n\varepsilon ]$), we see that for $\alpha \in (1,2]$
\begin{align}
\limsup_{n\rightarrow \infty }\sup_{x\in \mathcal{G}_{n}}\frac{c_{n}\bar{R}%
_{\varepsilon n}(x)}{H(an+x)}& \leq C(\varepsilon )\limsup_{n\rightarrow
\infty }\sup_{x\in \mathcal{G}_{n}}\frac{an+x}{c_{n}}  \notag \\
& =C(\varepsilon )\limsup_{n\rightarrow \infty }\delta _{n}=0.  \label{Al2}
\end{align}%
Similarly, writing $c^{-1}(n)$ for the inverse function of $c_{n}$ we
conclude by (\ref{new2}) (with $[n/2]$ replaced by $n-[n\varepsilon ]$) that
for $\alpha \in (0,1)$ and every $\gamma <1/2$.
\begin{eqnarray}
\limsup_{n\rightarrow \infty }\sup_{x\in \mathcal{G}_{n}}\frac{c_{n}\bar{R}%
_{\varepsilon n}(x)}{H(an+x)} &\leq &C(\varepsilon )\limsup_{n\rightarrow
\infty }\Bigl(\frac{N_{\delta _{n}c_{n}}}{n}\Bigr)^{1-2\gamma }  \notag \\
&=&C(\varepsilon )\limsup_{n\rightarrow \infty }\Bigl(\frac{c^{-1}(\delta
_{n}c_{n})}{c^{-1}(c_{n})}\Bigr)^{1-2\gamma }=0  \label{Al1}
\end{eqnarray}

According to the Gnedenko local limit theorem
\begin{equation*}
\sup_{k\in \lbrack 1,n(1-\varepsilon )]}\sup_{y\in \mathcal{I}%
_{x}(k,n)}\left\vert c_{n-k}\mathbf{P}(\bar{S}_{n-k}=y)-g_{\alpha ,\beta
}(0)\right\vert \rightarrow 0\quad \text{as }n\rightarrow \infty .
\end{equation*}%
Therefore,
\begin{align*}
\sum_{y\in \mathcal{I}_{x}(k,n)}\bar{b}_{k}(x-y)\mathbf{P}(\bar{S}_{n-k}& =y)
\\
=& \frac{g_{\alpha ,\beta }(0)+\Delta _{1}\left( x,n-k\right) }{c_{n-k}}%
\sum_{y\in \mathcal{I}_{x}(k,n)}\bar{b}_{k}(x-y),
\end{align*}%
where $\Delta _{1}\left( x,n-k\right) \rightarrow 0$ as $n\rightarrow \infty
$ uniformly in $x\in \mathcal{G}_{n}$ and $k\in \lbrack 1,n(1-\varepsilon )]$%
. Hence, by the identity
\begin{equation*}
\sum_{y\in \mathcal{I}_{x}(k,n)}\bar{b}_{k}(x-y)=B_{k}(a(n-k)+x),
\end{equation*}%
we see that
\begin{equation}
R_{\varepsilon n}(x)=(g_{\alpha ,\beta }(0)+\Delta _{2}(x,n))\left( \frac{1}{%
c_{n}}+\sum_{k=1}^{\left[ \varepsilon n\right] }\frac{1}{c_{n-k}}%
B_{k}(a(n-k)+x)\right) ,  \label{S3}
\end{equation}%
where $\Delta _{2}(x,n)\rightarrow 0$ as $n\rightarrow \infty $ uniformly in
$x\in \mathcal{G}_{n}$. Since the sequence $\left\{ c_{n},n\geq 1\right\} $
is non-decreasing and varies regularly with index $1/\alpha $ as $%
n\rightarrow \infty ,$ we have
\begin{align}
\sum_{k=1}^{[\varepsilon n]}B_{k}(a(n-k)+x)& \leq
c_{n}\sum_{k=n-[\varepsilon n]}^{n-1}\frac{1}{c_{k}}B_{n-k}(ak+x)  \notag \\
& \leq \left( (1-\varepsilon )^{-1/\alpha }+\Delta _{3}(x,n)\right)
\sum_{k=1}^{[\varepsilon n]}B_{k}(a(n-k)+x)  \label{S4}
\end{align}%
where $\Delta _{3}(x,n)\rightarrow 0$ as $n\rightarrow \infty $ uniformly in
$x\in \mathcal{G}_{n}$. On the other hand, for all $x>-an$,
\begin{equation}
H(an+x)-\sum_{k=[\varepsilon n]+1}^{\infty }B_{k}(a\left( n-k\right) +x)\leq
1+\sum_{k=1}^{[\varepsilon n]}B_{k}(a(n-k)+x)\leq H(an+x).  \label{S5}
\end{equation}%
Applying (\ref{P2}) gives for some constant $C_{1}=C_{1}(\varepsilon )$
\begin{align}
\sum_{k=[\varepsilon n]+1}^{\infty }B_{k}(a\left( n-k\right) +x)& \leq
(an+x)H(an+x)\sum_{k=[\varepsilon n]+1}^{\infty }\frac{C}{kc_{k}}  \notag \\
& \leq C_{1}\frac{(an+x)H(an+x)}{c_{n}}\leq C_{1}\delta _{n}H(an+x)
\label{S6}
\end{align}%
for all $x\in \mathcal{G}_{n}$. From (\ref{S5}) and (\ref{S6}) we conclude
that
\begin{equation}
\frac{1}{H(an+x)}\left( 1+\sum_{k=1}^{[\varepsilon n]}B_{k}(a(n-k)+x)\right)
-1\rightarrow 0  \label{S7}
\end{equation}%
uniformly in $x\in \mathcal{G}_{n}$. Combining (\ref{S3}), (\ref{S4}), and (%
\ref{S7}) leads to
\begin{equation*}
\limsup_{n\rightarrow \infty }\sup_{x\in \mathcal{G}_{n}}\left\vert \frac{%
c_{n}R_{\varepsilon n}(x)}{H(an+x)}-g_{\alpha ,\beta }(0)\right\vert \leq
r(\varepsilon ),
\end{equation*}%
where $r(\varepsilon )\rightarrow 0$ as $\varepsilon \rightarrow 0$. This
estimate, (\ref{Al2}), and (\ref{Al1}) \ show that
\begin{equation*}
\limsup_{n\rightarrow \infty }\sup_{x\in \mathcal{G}_{n}}\left\vert \frac{%
c_{n}n}{H(an+x)}\bar{b}_{n}(x)-g_{\alpha ,\beta }(0)\right\vert \leq
r(\varepsilon ).
\end{equation*}

Letting $\varepsilon \rightarrow 0$ and recalling that
\begin{equation*}
\bar{b}_{n}(x)=\mathbf{P}(S_{n}=an+x|\tau ^{-}>n)\mathbf{P}(\tau ^{-}>n)
\end{equation*}
we finish the proof of Theorem \ref{SmallDev}.

\subsection{Non-lattice case: Proof of Theorem~\protect\ref{SmallDev'}}

As in the proof of Theorem~\ref{NormalDev'} we restrict our attention to the
case $\Delta =1$. Some of our subsequent arguments are similar to those used
in the proof of Theorem~\ref{SmallDev}, and we skip the respective  details.

Using (\ref{n5}), (\ref{n6}) and Lemma \ref{Concentration} gives (in the
notation introduced after formula (\ref{c1}))
\begin{eqnarray*}
R_{\varepsilon }^{(1)}(x)+R_{\varepsilon }^{(2)}(x)
&=&\sum_{k=1}^{[(1-\varepsilon )n]}\int_{0}^{x}\mathbf{P}(S_{k}\in \lbrack
x-y,x-y+1))dB_{n-k}(y) \\
&\leq &C(\varepsilon )\frac{H(x)}{nc_{n}}\sum_{k=1}^{[(1-\varepsilon )n]}%
\mathbf{P}(0\leq S_{k}<x+1).
\end{eqnarray*}

By the arguments mimicing those used in the lattice case one can easily show
that
\begin{equation}
\lim_{n\rightarrow \infty }\sup_{0<x\leq \delta _{n}c_{n}}\frac{c_{n}}{H(x)}%
\left( R_{\varepsilon }^{(1)}(x)+R_{\varepsilon }^{(2)}(x)\right) =0.
\label{Y1}
\end{equation}%
Further, by the Stone local limit theorem
\begin{equation*}
\int_{0}^{x}\mathbf{P}(S_{k}\in \lbrack x-y,x-y+1))dB_{n-k}(y)=\frac{%
g_{\alpha ,\beta }(0)+\Delta _{1}(k,x)}{c_{k}}B_{n-k}(x),
\end{equation*}%
where $\Delta _{1}(k,x)\rightarrow 0$ uniformly in $x\in (0,\delta
_{n}c_{n}) $ and $k\in \lbrack (1-\varepsilon )n,n]$. Therefore,
\begin{align*}
R_{\varepsilon }^{(3)}(x)& =\mathbf{P}(S_{n}\in \lbrack x,x+1))+\hspace{-2mm}%
\sum_{k=[(1-\varepsilon )n]+1}^{n-1}\int_{0}^{x}\mathbf{P}(S_{k}\in \lbrack
x-y,x-y+1))dB_{n-k}(y) \\
& =(g_{\alpha ,\beta }(0)+\Delta _{2}(n,x))\Bigl(\frac{1}{c_{n}}%
+\sum_{k=1}^{[\varepsilon n]}\frac{1}{c_{n-k}}B_{k}(x)\Bigr),
\end{align*}%
where $\Delta _{2}(n,x)\rightarrow 0$ uniformly in $x\in (0,\delta
_{n}c_{n}) $. Therefore, as in the lattice case,
\begin{equation}
\limsup_{n\rightarrow \infty }\sup_{0<x\leq \delta _{n}c_{n}}\Bigl|\frac{%
c_{n}}{H(x)}R_{\varepsilon }^{(3)}(x)-g_{\alpha ,\beta }(0)\Bigr|\leq
r(\varepsilon ),  \label{Y2}
\end{equation}%
where $r(\varepsilon )\rightarrow 0$ as $\varepsilon \rightarrow 0$.
Combining (\ref{Y1}) and (\ref{Y2}), we get
\begin{equation}
\limsup_{n\rightarrow \infty }\sup_{0<x\leq \delta _{n}c_{n}}\Bigl|\frac{%
c_{n}}{H(x)}\Bigl(R_{\varepsilon }^{(1)}(x)+R_{\varepsilon
}^{(2)}(x)+R_{\varepsilon }^{(3)}(x)\Bigr)-g_{\alpha ,\beta }(0)\Bigr|=0.
\label{Y3}
\end{equation}

Now using definition (\ref{c1}) we write $R^{(0)}(x)=$ $R_{\varepsilon
}^{(4)}(x)+R_{\varepsilon }^{(5)}(x),$ where%
\begin{equation*}
R_{\varepsilon }^{(4)}(x):=\sum_{k=1}^{[(1-\varepsilon )n]}\int_{x}^{x+1}%
\mathbf{P}(S_{k}\in (0,x-y+1))dB_{n-k}(y)
\end{equation*}%
and%
\begin{equation*}
R_{\varepsilon }^{(5)}(x):=\sum_{k=[(1-\varepsilon )n]+1}^{n-1}\int_{x}^{x+1}%
\mathbf{P}(S_{k}\in (0,x-y+1))dB_{n-k}(y).
\end{equation*}%
Evidently,
\begin{equation*}
R_{\varepsilon }^{(4)}(x)\leq \sum_{k=1}^{[(1-\varepsilon )n]}\mathbf{P}%
(S_{k}\in (0,1))b_{n-k}(x).
\end{equation*}%
Applying (\ref{EstS1}) and (\ref{P1}), we see that
\begin{equation*}
R_{\varepsilon }^{(4)}(x)\leq H(x)\sum_{k=1}^{[(1-\varepsilon )n]}\frac{1}{%
c_{k}}\frac{1}{(n-k)c_{n-k}}\leq \frac{C(\varepsilon )}{nc_{n}}\sum_{k=1}^{n}%
\frac{1}{c_{k}}
\end{equation*}%
Observing that $\sum_{k=1}^{n}c_{k}^{-1}\leq C(1+n/c_{n})$, we conclude that
\begin{equation}
\lim_{n\rightarrow \infty }\sup_{0<x\leq \delta _{n}c_{n}}\frac{c_{n}}{H(x)}%
R_{\varepsilon }^{(4)}(x)=0.  \label{Y4}
\end{equation}%
Further, by the Stone local limit theorem,
\begin{align*}
\int_{x}^{x+1}\mathbf{P}(S_{k}& \in \lbrack 0,x-y+1))dB_{n-k}(y)\hspace{3cm}
\\
& =\frac{(g_{\alpha ,\beta }(0)+\Delta _{3}(k,x))}{c_{k}}%
\int_{x}^{x+1}(x-y+1)dB_{n-k}(y),
\end{align*}%
where $\Delta _{3}(k,x)\rightarrow 0$ uniformly in $x\in (0,\delta
_{n}c_{n}] $ and $k\in \lbrack (1-\varepsilon )n,n]$. Integration by parts
gives
\begin{equation*}
\int_{x}^{x+1}(x-y+1)dB_{n-k}(y)=-B_{n-k}(x)+\int_{x}^{x+1}B_{n-k}(y)dy.
\end{equation*}%
Consequently,
\begin{equation}
R_{\varepsilon }^{(5)}(x)=(g_{\alpha ,\beta }(0)+\Delta
_{4}(n,x))\sum_{k=1}^{[\varepsilon n]}\frac{1}{c_{n-k}}\Bigl(%
\int_{x}^{x+1}B_{k}(y)dy-B_{k}(x)\Bigr),  \label{Y5}
\end{equation}%
where $\Delta _{4}(n,x)\rightarrow 0$ uniformly in $x\in (0,\delta
_{n}c_{n}] $.

Setting
\begin{equation*}
I(x):=\int_{x}^{x+1}H(y)dy-H(x)
\end{equation*}%
we see, similarly to the proof in the lattice case, that
\begin{equation}
\limsup_{n\rightarrow \infty }\sup_{0<x\leq \delta _{n}c_{n}}\Big|\frac{c_{n}%
}{I(x)}\sum_{k=1}^{[\varepsilon n]}\frac{1}{c_{n-k}}\Bigl(%
\int_{x}^{x+1}B_{k}(y)dy-B_{k}(x)\Bigr)-g_{\alpha ,\beta }(0)\Big|\leq
r(\varepsilon ),  \label{Y6}
\end{equation}%
where $r(\varepsilon )\rightarrow 0$ as $\varepsilon \rightarrow 0$. By (\ref%
{Y4}) -- (\ref{Y6}) we deduce
\begin{equation}
\lim_{n\rightarrow \infty }\sup_{0<x\leq \delta _{n}c_{n}}\Big|\frac{c_{n}}{%
I(x)}R^{(0)}(x)-g_{\alpha ,\beta }(0)\Bigr|=0.  \label{Y7}
\end{equation}%
Substituting (\ref{Y3}) and (\ref{Y7}) into (\ref{c1}) finishes the proof.

\subsection{Proof of Theorem \protect\ref{Density}\label{Sden}}

It is sufficient to show that there exists a constant $C>0$ such that
\begin{equation}
p_{\alpha ,\beta }(\varepsilon _{m})\sim C\varepsilon _{m}^{\alpha \rho }%
\text{ as }m\rightarrow \infty  \label{c0}
\end{equation}%
for every sequence $\varepsilon _{m}\rightarrow 0$. Since $H(x)$ is
regularly varying with index $\alpha \rho $, there exists a sequence $%
n_{1}(m)\rightarrow \infty $ as $m\rightarrow \infty $ such that
\begin{equation*}
\sup_{n\geq n_{1}(m)}\left\vert \frac{\varepsilon _{m}^{\alpha \rho }H(c_{n})%
}{H(\varepsilon _{m}c_{n})}-1\right\vert \rightarrow 0\text{ as }%
m\rightarrow \infty .
\end{equation*}

From this fact and Theorem~\ref{SmallDev'} we deduce:
\begin{eqnarray}
c_{n}\mathbf{P}(S_{n}\in \lbrack \varepsilon _{m}c_{n},\varepsilon
_{m}c_{n}+1)|\tau ^{-}>n) &=&g_{\alpha ,\beta }(0)\frac{H(\varepsilon
_{m}c_{n})}{n\mathbf{P}\left( \tau ^{-}>n\right) }(1+\varphi _{n,m}^{(1)})
\notag \\
&=&g_{\alpha ,\beta }(0)\frac{\varepsilon _{m}^{\alpha \rho }H(c_{n})}{n%
\mathbf{P}\left( \tau ^{-}>n\right) }(1+\varphi _{n,m}^{(2)})  \label{D1}
\end{eqnarray}%
where, for $i=1,2$,%
\begin{equation*}
\sup_{n\geq n_{1}(m)}|\varphi _{n,m}^{(i)}|\rightarrow 0\text{ as }%
m\rightarrow \infty .
\end{equation*}%
Further, according to Theorem~\ref{NormalDev'},
\begin{equation}
c_{n}\mathbf{P}(S_{n}\in \lbrack \varepsilon _{m}c_{n},\varepsilon
_{m}c_{n}+1)|\tau ^{-}>n)=p_{\alpha ,\beta }(\varepsilon _{m})+\varphi
_{n,m},  \label{D2}
\end{equation}%
where $\varphi _{n,m}=\varphi _{n,m}(\varepsilon _{m})\rightarrow 0$ as $%
n\rightarrow \infty $ uniformly for all possible choices of $\varepsilon
_{m} $, that is,
\begin{equation}
\sup_{\{\varepsilon _{m}\}}|\varphi _{n,m}|\leq \Phi _{n}\quad \text{and}%
\quad \lim_{n\rightarrow \infty }\Phi _{n}=0.  \label{D2'}
\end{equation}

Comparing (\ref{D1}) and (\ref{D2}) gives%
\begin{equation}
p_{\alpha ,\beta }(\varepsilon _{m})=g_{\alpha ,\beta }(0)\frac{\varepsilon
_{m}^{\alpha \rho }H(c_{n(m)})}{n(m)\mathbf{P}\left( \tau ^{-}>n(m)\right) }%
(1+\varphi _{n(m),m}^{(2)})-\varphi _{n(m),m}  \label{D3}
\end{equation}%
where $n(m)$ is any sequence satisfying $n(m)\geq n_{1}(m)$ for all $m\geq 1$%
. Let $n_{2}(m)$ be defined by the relation
\begin{equation*}
n_{2}(m):=\min \{n\geq 1:\sup_{k\geq n}\Phi _{n}<\varepsilon _{m}^{\alpha
\rho +1}\}.
\end{equation*}%
Now, if $n(m)\geq \max \{n_{1}(m),n_{2}(m)\}$, then from the definition of $%
n_{2}(m)$ and (\ref{D3}) we have
\begin{equation*}
p_{\alpha ,\beta }(\varepsilon _{m})=g_{\alpha ,\beta }(0)\varepsilon
_{m}^{\alpha \rho }\frac{H(c_{n(m)})}{n(m)\mathbf{P}(\tau ^{-}>n(m))}%
(1+\varphi _{n(m),m}^{(2)})+O(\varepsilon _{m}^{\alpha \rho +1}).
\end{equation*}%
Taking into account (\ref{AsH}), we obtain (\ref{c0}). The theorem is proved.


\section{Proof of Theorem \protect\ref{LadEpoch}}

We start with the following technical lemma which may be known from the
literature.

\begin{lemma}
\label{Ltec}Let $w(n)$ be a monotone increasing function. If, \ for some $%
\gamma>0,$ there exist slowly varying functions $l^{\ast}(n)$ and $%
l^{\ast\ast}(n)$ such that, as $n\rightarrow\infty$,%
\begin{equation*}
\sum_{k=n}^{\infty}\frac{w(k)}{k^{\gamma+1}l^{\ast}(k)}\sim\frac{1}{%
n^{\gamma }l^{\ast\ast}(n)},
\end{equation*}
then, as $n\rightarrow\infty$,%
\begin{equation*}
w(n)\sim\gamma\frac{l^{\ast}(n)}{l^{\ast\ast}(n)}.
\end{equation*}
\end{lemma}

\begin{proof}
Let, for this lemma only, $r_{i}(n),n=1,2,...;i=1,2,3,4$ be sequences of
real numbers vanishing as $n\rightarrow \infty .$ For $\delta \in (0,1)$ we
have by monotonicity of $w(n)$ and properties of slowly varying functions%
\begin{align*}
w(\left[ \delta n\right] )\sum_{k=\left[ \delta n\right] }^{n}\frac{1}{%
k^{\gamma +1}l^{\ast }(k)}& =w(\left[ \delta n\right] )\frac{1+r_{2}(n)}{%
\gamma n^{\gamma }l^{\ast }(n)}\left( \delta ^{-\gamma }-1\right) \\
& \leq \sum_{k=\left[ \delta n\right] }^{n}\frac{w(k)}{k^{\gamma +1}l^{\ast
}(k)}=\frac{1+r_{1}(n)}{n^{\gamma }l^{\ast \ast }(n)}\left( \delta ^{-\gamma
}-1\right) \\
& \leq w(n)\sum_{k=\left[ \delta n\right] }^{n}\frac{1}{k^{\gamma +1}l^{\ast
}(k)} \\
& =w(n)\frac{1+r_{2}(n)}{\gamma n^{\gamma }l^{\ast }(n)}\left( \delta
^{-\gamma }-1\right) .
\end{align*}%
Hence it follows that%
\begin{equation*}
w(\left[ \delta n\right] )\leq \frac{1+r_{1}(n)}{1+r_{2}(n)}\frac{\gamma
l^{\ast }(n)}{l^{\ast \ast }(n)}\leq w(n)
\end{equation*}%
and, therefore,%
\begin{equation*}
\frac{1+r_{1}(n)}{1+r_{2}(n)}\frac{\gamma l^{\ast }(n)}{l^{\ast \ast }(n)}%
\leq w(n)\leq \frac{1+r_{3}(\left[ n\delta ^{-1}\right] )}{1+r_{4}(\left[
n\delta ^{-1}\right] )}\frac{\gamma l^{\ast }(\left[ n\delta ^{-1}\right] )}{%
l^{\ast \ast }(\left[ n\delta ^{-1}\right] )}.
\end{equation*}%
Since $l^{\ast }$ and $l^{\ast \ast }$ are slowly varying functions, we get%
\begin{equation*}
\lim_{n\rightarrow \infty }\frac{w(n)l^{\ast \ast }(n)}{\gamma l^{\ast }(n)}%
=1,
\end{equation*}%
as desired.
\end{proof}

\begin{remark}
\label{Rem22}By the same arguments one can show that if $w(x)$ is a monotone
increasing function and,\ for some $\gamma >0,$ there exist slowly varying
functions $l^{\ast }(x)$ and $l^{\ast \ast }(x)$ such that, as $x\rightarrow
\infty $,%
\begin{equation*}
\int_{x}^{\infty }\frac{w(y)dy}{y^{\gamma +1}l^{\ast }(y)}\sim \frac{1}{%
x^{\gamma }l^{\ast \ast }(x)},
\end{equation*}%
then, as $x\rightarrow \infty $,%
\begin{equation*}
w(x)\sim \gamma \frac{l^{\ast }(x)}{l^{\ast \ast }(x)}.
\end{equation*}
Note also that this staement for the case $l^{\ast }(x)\equiv Const$ can be
found in \cite[Chapter VIII, Section 9]{FE}.
\end{remark}

\subsection{Proof of Theorem \protect\ref{LadEpoch} for $\{0<\protect\alpha %
<2,\,\protect\beta <1\}$.}

For a fixed $\varepsilon \in (0,1)$ write%
\begin{equation*}
\mathbf{P}\left( \tau ^{-}=n\right) =\mathbf{P}\left( S_{n}\leq 0;\tau
^{-}>n-1\right) =:J_{1}(\varepsilon c_{n})+J_{2}(\varepsilon c_{n})
\end{equation*}%
where%
\begin{equation*}
J_{1}(\varepsilon c_{n}):=\int_{\varepsilon }^{\infty }\mathbf{P}\left(
X\leq -yc_{n}\right) \mathbf{P}\left( S_{n-1}\in c_{n}dy;\tau
^{-}>n-1\right) .
\end{equation*}%
and%
\begin{equation*}
J_{2}(\varepsilon c_{n}):=\int_{0}^{\varepsilon c_{n}}\mathbf{P}\left( X\leq
-y\right) \mathbf{P}\left( S_{n-1}\in dy;\tau ^{-}>n-1\right)
\end{equation*}
First we study properties of $J_{1}(\varepsilon c_{n})$.

We know from (\ref{Tailtwo}) and (\ref{tailF}) that if $X\in \mathcal{D}%
\left( \alpha ,\beta \right) $ with $0<\alpha <2$ and $\beta <1$, then, for
a $q\in (0,1]$,
\begin{equation}
\mathbf{P}\left( X\leq -y\right) \sim \frac{q}{y^{\alpha }l_{0}(y)}\quad
\text{as }y\rightarrow \infty ,  \label{lefttail}
\end{equation}%
and, according to (\ref{tailF'}),%
\begin{equation*}
\mathbf{P}\left( X\leq -c_{n}\right) \sim \frac{q(2-\alpha )}{\alpha n}\quad
\text{as }n\rightarrow \infty .
\end{equation*}%
Moreover, for any $\varepsilon >0$,
\begin{equation}
\frac{\mathbf{P}\left( X\leq -yc_{n}\right) }{\mathbf{P}\left( X\leq
-c_{n}\right) }\rightarrow y^{-\alpha }\quad \text{as }n\rightarrow \infty ,
\label{aa}
\end{equation}%
uniformly in $y\in (\varepsilon ,\infty ).$

It easily follows from (\ref{aa}) and (\ref{meander1}) that, as $%
n\rightarrow \infty ,$%
\begin{align}
& J_{1}(\varepsilon c_{n})=\mathbf{P}\left( X\leq -c_{n}\right) \mathbf{P}%
\left( \tau ^{-}>n-1\right) \int_{\varepsilon }^{\infty }\frac{\mathbf{P}%
\left( X\leq -yc_{n}\right) }{\mathbf{P}\left( X\leq -c_{n}\right) }\mathbf{P%
}\left( \frac{S_{n-1}}{c_{n}}\in dy\,|\,\tau ^{-}>n-1\right)  \notag \\
& \sim \frac{q(2-\alpha )l(n)}{\alpha n^{2-\rho }}\int_{\varepsilon
}^{\infty }\frac{\mathbf{P}\left( X\leq -yc_{n}\right) }{\mathbf{P}\left(
X\leq -c_{n}\right) }\mathbf{P}\left( \frac{S_{n-1}}{c_{n}}\in dy\,|\,\tau
^{-}>n-1\right)  \notag \\
& \sim \frac{q(2-\alpha )l(n)}{\alpha n^{2-\rho }}\int_{\varepsilon
}^{\infty }\frac{\mathbf{P}\left( M_{\alpha ,\beta }\in dy\,\right) }{%
y^{\alpha }}.  \label{exact1}
\end{align}

From Theorem~\ref{Density} follows that $p_{\alpha,\beta}(y)\leq
Cy^{\alpha\rho}$ for some positive constant $C$ and all $y\in(0,1]$.
Consequently,
\begin{equation*}
\int_0^\infty\frac{\mathbf{P}\left(M_{\alpha,\beta }\in dy\,\right) }{%
y^{\alpha}} \leq C\int_0^1 y^{-\alpha+\alpha\rho}dy+\mathbf{P}%
\left(M_{\alpha,\beta }>1\,\right).
\end{equation*}
Noting that the condition $\beta<1$ implies the bound $-\alpha+\alpha\rho>-1$%
, we conclude that
\begin{equation*}
\int_0^\infty\frac{\mathbf{P}\left(M_{\alpha,\beta }\in dy\,\right) }{%
y^{\alpha}} <\infty.
\end{equation*}

Therefore,
\begin{equation}
\lim_{\varepsilon \rightarrow 0}\lim_{n\rightarrow \infty }\frac{\alpha
n^{2-\rho }}{q(2-\alpha )l(n)}J_{1}(\varepsilon c_{n})=\int_{0}^{\infty }%
\frac{\mathbf{P}\left( M_{\alpha ,\beta }\in dy\,\right) }{y^{\alpha }}.
\label{exact2}
\end{equation}%
Now to complete the proof of Theorem \ref{LadEpoch} in the case $\{0<\alpha
<2,\beta <1\}$ it remains to demonstrate that
\begin{equation}
\lim_{\varepsilon \rightarrow 0}\limsup_{n\rightarrow \infty }\frac{%
n^{2-\rho }}{l(n)}J_{2}(\varepsilon c_{n})=0.  \label{cc}
\end{equation}%
To this aim we observe that
\begin{equation*}
J_{2}(\varepsilon c_{n})\leq \sum_{j=0}^{[\varepsilon c_{n}]+1}\mathbf{P}%
\left( X\leq -j\right) b_{n-1}(j)=:R(\varepsilon c_{n})
\end{equation*}%
and evaluate $R(\varepsilon c_{n})$ separately for the following two cases:

(i) $\beta\in(-1,1);$

(ii) $\beta=-1.$

\vspace{12pt}

\textbf{(i)}. In view of (\ref{P1}), equivalences (\ref{RenStand}) and (\ref%
{Tailtwo}), we have
\begin{align*}
R(\varepsilon c_{n})& \leq C\sum_{j=1}^{[\varepsilon c_{n}]+1}\frac{1}{%
j^{\alpha }l_{0}(j)}\frac{j^{\alpha \rho }l_{2}(j)}{nc_{n}}\leq C_{2}\frac{1%
}{nc_{n}}\left( \varepsilon c_{n}\right) ^{1-\alpha (1-\rho )}\frac{%
l_{2}(\varepsilon c_{n})}{l_{0}(\varepsilon c_{n})} \\
& \leq C_{3}\frac{1}{nc_{n}}\varepsilon ^{1-\alpha (1-\rho )-\gamma
}c_{n}^{1-\alpha (1-\rho )}\frac{l_{2}(c_{n})}{l_{0}(c_{n})}\leq
C_{4}\varepsilon ^{1-\alpha (1-\rho )-\gamma }\frac{H(c_{n})\mathbf{P}%
(|X|>c_{n})}{n}
\end{align*}%
for any fixed $\gamma \in (0,1-\alpha (1-\rho ))$ and all sufficiently large
$n$. At the third step we have applied (\ref{prop}) to the function $%
l_{2}(x)/l_{0}(x)$. Using (\ref{tailF'}) and (\ref{AsH}), we get
\begin{equation*}
R(\varepsilon c_{n})\leq C\varepsilon ^{1-\alpha (1-\rho )-\gamma }\frac{%
\mathbf{P}(\tau ^{-}>n)}{n}.
\end{equation*}

Hence on account of (\ref{integ1}) we conclude that%
\begin{equation}
R(\varepsilon c_{n})\leq C\frac{l(n)}{n^{2-\rho }}\varepsilon ^{1-\alpha
(1-\rho )-\gamma }.  \label{Remalpha21}
\end{equation}

\textbf{(ii)}. It follows from (\ref{ro}) that if $\beta =-1$, then $\alpha
\rho =1.$ By Lemma \ref{Renew2}, $H(x)\leq Cxl_{3}(x)$. Combining this
estimate with (\ref{P1}) \ yields
\begin{equation*}
b_{n}(j)\leq C\frac{jl_{3}(j)}{nc_{n}}.
\end{equation*}%
Recalling (\ref{lefttail}) \ and using (\ref{prop}), we obtain for any fixed
$\gamma \in (0,2-\alpha )$ and all $n\geq n(\gamma )$,%
\begin{align}
R(\varepsilon c_{n})& \leq C\sum_{j=0}^{[\varepsilon c_{n}]+1}\mathbf{P}%
\left( X\leq -j\right) \frac{jl_{3}(j)}{nc_{n}}  \notag \\
& \leq C_{1}\left( \varepsilon c_{n}\right) ^{2-\alpha }\frac{1}{nc_{n}}%
\frac{l_{3}(\varepsilon c_{n})}{l_{0}(\varepsilon c_{n})}  \notag \\
& \leq C_{2}\varepsilon ^{2-\alpha -\gamma }\frac{1}{n}\frac{%
c_{n}l_{3}(c_{n})}{c_{n}^{\alpha }l_{0}(c_{n})}\leq C_{3}\varepsilon
^{2-\alpha -\gamma }\frac{l(n)}{n^{2-\rho }},  \label{Remalpha3}
\end{align}%
where at the last step we have applied the inequalities $H(c_{n})\leq
Cc_{n}l_{3}(c_{n})\leq Cn^{\rho }l(n),$ \ following from (\ref{RenewRelat}),
(\ref{AsH}), and (\ref{integ1}), and the relation%
\begin{equation*}
\frac{1}{n}\sim \frac{\alpha }{2-\alpha }\frac{1}{c_{n}^{\alpha }l_{0}(c_{n})%
},
\end{equation*}%
being a corollary of (\ref{UU}).

Estimates (\ref{Remalpha21}) and (\ref{Remalpha3}) imply (\ref{cc}).
Combining (\ref{exact2}) with (\ref{cc}) leads to
\begin{equation}
\mathbf{P}\left( \tau ^{-}=n\right) \sim \frac{q(2-\alpha )l(n)}{\alpha
n^{2-\rho }}\int_{0}^{\infty }\frac{\mathbf{P}\left( M_{\alpha ,\beta }\in
dy\,\right) }{y^{\alpha }}=\frac{q(2-\alpha )l(n)}{\alpha n^{2-\rho }}%
\mathbf{E}\left( M_{\alpha ,\beta }\right) ^{-\alpha }.  \label{local}
\end{equation}%
Summation over $n$ gives%
\begin{equation*}
\mathbf{P}\left( \tau ^{-}>n\right) =\sum_{k=n+1}^{\infty }\mathbf{P}\left(
\tau ^{-}=k\right) \sim \frac{q(2-\alpha )}{\alpha \left( 1-\rho \right) }%
\frac{l(n)}{n^{1-\rho }}\mathbf{E}\left( M_{\alpha ,\beta }\right) ^{-\alpha
}.
\end{equation*}%
Comparing this with (\ref{integ1}), we get an interesting identity%
\begin{equation}
\mathbf{E}\left( M_{\alpha ,\beta }\right) ^{-\alpha }=\alpha (1-\rho
)/q(2-\alpha )  \label{IDENtity}
\end{equation}%
which, in view of (\ref{local}), completes the proof of Theorem \ref%
{LadEpoch} for $0<\alpha <2,\,\beta <1$.


\subsection{Proof of Theorem \protect\ref{LadEpoch} for $\{1<\protect\alpha%
<2,\protect\beta=1\}\cup\{\protect\alpha=2,\protect\beta=0\}$}

We consider only the lattice random walks with $a\in (0,1)$ and $h=1$. The
non-lattice case requires only minor changes. The main reason for the choice
of the lattice situation is the fact that only in this case we can get
oscillating sequences $Q_{n}^{-}$.

By the total probability formula,
\begin{equation}
\mathbf{P}(\tau ^{-}=n+1)=\sum_{k>-an}\mathbf{P}(S_{n}=an+k;\tau ^{-}>n)%
\mathbf{P}(X\leq -an-k).  \label{A1}
\end{equation}%
One can easily verify that under the conditions imposed on the distribution
of $X$ there exists a sequence $\delta _{n}\rightarrow 0$ such that $\delta
_{n}c_{n}\rightarrow \infty $ and
\begin{equation}
\mathbf{P}(X\leq -\delta _{n}c_{n})=o(n^{-1})\text{ as }n\rightarrow \infty .
\label{A2}
\end{equation}%
Using, as earlier, the notation $\mathcal{G}_{n}=(-an,-an+\delta
_{n}c_{n})\cap \mathbb{Z}$, and combining (\ref{A1}) with (\ref{A2}), we
obtain
\begin{equation*}
\mathbf{P}(\tau ^{-}=n+1)=\sum_{k\in \mathcal{G}_{n}}\mathbf{P}%
(S_{n}=an+k;\tau ^{-}>n)\mathbf{P}(X\leq -an-k)+o\Bigl(\frac{l(n)}{n^{2-\rho
}}\Bigr).
\end{equation*}%
Let $\left\{ an\right\} $ be the fractional part of $an$. By Theorem~\ref%
{SmallDev}
\begin{align}
\mathbf{P}(\tau ^{-}& =n+1)=\frac{g_{\alpha ,\beta }(0)+o(1)}{nc_{n}}%
\sum_{k\in \mathcal{G}_{n}}H(an+k)\mathbf{P}(X\leq -an-k)+o\Bigl(\frac{l(n)}{%
n^{2-\rho }}\Bigr)  \notag  \label{A3} \\
& =\frac{g_{\alpha ,\beta }(0)+o(1)}{nc_{n}}\sum_{j=0}^{\delta
_{n}c_{n}}H(\left\{ an\right\} +j)\mathbf{P}(X\leq -\left\{ an\right\} -j)+o%
\Bigl(\frac{l(n)}{n^{2-\rho }}\Bigr).
\end{align}

For $z\geq 0$ set%
\begin{equation*}
\omega (z;n):=\sum_{j=0}^{\delta _{n}c_{n}}H(z+j)\mathbf{P}(X\leq
-z-j),\qquad \omega (n):=\omega (0;n),
\end{equation*}%
and using the equality
\begin{equation}
\mathbf{E}(-S_{\tau ^{-}})=\int_{0}^{\infty }H(x)\mathbf{P}(X\leq -x)dx
\label{BounExp}
\end{equation}%
(see Doney \cite{Don82}) consider the "if" part of Theorem \ref{LadEpoch}
under the hypotheses of points (a), (b), and (c) \textbf{\ } separately.

\textbf{(a)} Condition $E(-S_{\tau ^{-}})=\infty $ implies
\begin{equation}
\omega (n)\rightarrow \infty \text{ as }n\rightarrow \infty .  \label{A4}
\end{equation}%
Since $H(u)$ is a renewal function, there exists a constant $C$ such that
\begin{equation}
H(u+v)-H(u)\leq C(v+1)\text{ for all }u,v\geq 0.  \label{RenF}
\end{equation}%
By (\ref{RenF}) and monotonicity of $H(u)$ and $\mathbf{P}(X\leq -u)$ we
conclude that
\begin{equation*}
\omega (\left\{ an\right\} ;n)\leq \sum_{j=0}^{\delta _{n}c_{n}}H(j+1)%
\mathbf{P}(X\leq -j)\leq \omega (n)+C\sum_{j=0}^{\delta _{n}c_{n}}\mathbf{P}%
(X\leq -j)
\end{equation*}%
and
\begin{align*}
\omega (\left\{ an\right\} ;n)& \geq \sum_{j=0}^{\delta _{n}c_{n}}H(j)%
\mathbf{P}(X\leq -j-1)\geq \omega (1;n)-C\sum_{j=0}^{\delta _{n}c_{n}}%
\mathbf{P}(X\leq -j) \\
& \geq \omega (n)-C\sum_{j=0}^{\delta _{n}c_{n}}\mathbf{P}(X\leq -j).
\end{align*}%
From (\ref{A4}) and the fact that $H(x)\rightarrow \infty $ as $x\rightarrow
\infty $ we deduce that
\begin{equation*}
\sum_{j=0}^{\varepsilon _{n}c_{n}}\mathbf{P}(X\leq -j)=o\left( \omega
(n)\right) \text{ as }n\rightarrow \infty .
\end{equation*}%
This yields $\omega (\left\{ an\right\} ;n)\sim \omega (n)$ as $n\rightarrow
\infty $ which, \ combined with (\ref{A3}), gives
\begin{equation}
\mathbf{P}(\tau ^{-}=n+1)=\frac{g_{\alpha ,\beta }(0)+o(1)}{nc_{n}}\omega
(n)+o\Bigl(\frac{l(n)}{n^{2-\rho }}\Bigr),\text{ }n\rightarrow \infty .
\label{A5}
\end{equation}%
Summing over $n\geq k$, we get, as $k\rightarrow \infty $,
\begin{equation*}
\frac{l(k)}{k^{1-\rho }}\sim \mathbf{P}(\tau ^{-}>k)=(g_{\alpha ,\beta
}(0)+o(1))\sum_{n=k}^{\infty }\frac{\omega (n)}{nc_{n}}+o\Bigl(\frac{l(k)}{%
k^{1-\rho }}\Bigr).
\end{equation*}%
We know from (\ref{ro})\ that $\rho =1-1/\alpha $ if $\{1<\alpha <2,\beta
=1\}$ or $\{\alpha =2,\beta =0\}.$ Since $\omega (n)$ is non-decreasing and,
by (\ref{asyma}), $c_{n}$ is regularly varying with index $1/\alpha $, Lemma~%
\ref{Ltec} implies
\begin{equation*}
\frac{\omega (n)}{nc_{n}}\sim \frac{1-\rho }{g_{\alpha ,\beta }(0)}\frac{l(n)%
}{n^{2-\rho }}\text{ as }n\rightarrow \infty .
\end{equation*}%
Consequently,
\begin{equation*}
\mathbf{P}(\tau ^{-}=n)=(1-\rho )\frac{l(n)}{n^{2-\rho }}(1+o(1)),\text{ }%
n\rightarrow \infty .
\end{equation*}%
This finishes the proof of (\ref{LE}) \ given $\mathbf{E}(-S_{\tau
^{-}})=\infty $.

\textbf{(b)} The assumption $\mathbf{E}(-S_{\tau ^{-}})<\infty $ and
relations (\ref{RenStand}), (\ref{RenewRelat}), and (\ref{BounExp})\ imply
\begin{equation*}
\sum_{j>\delta _{n}c_{n}}H(\left\{ an\right\} +j)\mathbf{P}(X\leq -\left\{
an\right\} -j)\rightarrow 0\text{ as }n\rightarrow \infty
\end{equation*}%
and, consequently,
\begin{equation*}
\omega (\left\{ an\right\} ;n)=\Omega (\left\{ an\right\} )+o(1)\text{ as }%
n\rightarrow \infty
\end{equation*}%
where%
\begin{equation*}
\Omega (\left\{ an\right\} ):=\sum_{j=0}^{\infty }H(\left\{ an\right\} +j)%
\mathbf{P}(X\leq -\left\{ an\right\} -j)
\end{equation*}%
Combining this representation with (\ref{A3}), observing that $\Omega
(\left\{ an\right\} )<C<\infty $ if $\mathbf{E}(-S_{\tau ^{-}})<\infty ,$
and recalling Lemma \ref{BehC} we see that
\begin{equation}
\mathbf{P}(\tau ^{-}=n+1)=\frac{g_{\alpha ,\beta }(0)}{nc_{n}}\Omega
(\left\{ an\right\} )+o\Bigl(\frac{l(n)}{n^{2-\rho }}\Bigr),\text{ }%
n\rightarrow \infty .  \label{A6}
\end{equation}

Denote
\begin{equation*}
\tilde{\Omega}(\left\{ an\right\} ):=\mathbf{P}(X\leq -\left\{ an\right\} )%
\mathrm{I}(\left\{ an\right\} >0)+\mathbf{P}(X\leq -1)\mathrm{I}(\left\{
an\right\} =0).
\end{equation*}%
Since $X$ is $(1,a)-$lattice, the quantity $\tilde{\Omega}(\left\{
an\right\} )$ is either $0$ or not less than some positive number $\tilde{%
\Omega}_{\ast }$. Furthermore, one can easily verify that $\Omega (\left\{
an\right\} )\geq \tilde{\Omega}(\left\{ an\right\} )$ and $\Omega (\left\{
an\right\} )=0$ if and only if $\tilde{\Omega}(\left\{ an\right\} )=0$.
Consequently, $\Omega (\left\{ an\right\} )$ is either\textbf{\ }zero or not
less than $\tilde{\Omega}_{\ast }$. Finally, in view of (\ref{A1}) $\tilde{%
\Omega}(\left\{ an\right\} )=0$ implies $\mathbf{P}(\tau ^{-}=n+1)=0$.
Therefore, we can rewrite (\ref{A6}) in the form
\begin{equation}
\mathbf{P}(\tau ^{-}=n+1)=\frac{g_{\alpha ,\beta }(0)}{nc_{n}}\Omega
(\left\{ an\right\} )(1+o(1)).  \label{A88}
\end{equation}

Now (\ref{A88}) and (\ref{A7}) give (\ref{LE}) with
\begin{equation}
Q_{n}^{-}:=C_{0}g_{\alpha ,\beta }(0)\Omega (\left\{ a(n-1)\right\} ).
\label{A89}
\end{equation}%
If $a=0$, then, evidently,
\begin{equation*}
Q_{n}^{-}\equiv C_{0}g_{\alpha ,\beta }(0)\Omega (0)=C_{0}g_{\alpha ,\beta
}(0)\mathbf{E}(-S_{\tau ^{-}}):=Q,
\end{equation*}%
and, consequently,
\begin{equation*}
\mathbf{P}(\tau ^{-}=n)=Q\frac{l(n)}{n^{2-\rho }}(1+o(1)).
\end{equation*}%
Comparing this asymptotic equality with the known tail behavior of the
distribution of $\tau ^{-}$, we infer that $Q$ should be equal to $1-\rho .$

This finishes the proof of (\ref{LE})under the conditions of point \textbf{%
(b)}.

To demonstrate the validity of \ (\ref{LE}) under the conditions of point
\textbf{(c)} one should made only evident minor changes of the just used
arguments and we omit the respective details.

To justify the "only if" part of Theorem \ref{LadEpoch} we need to show that
the sequence $\left\{ Q_{n}^{-},n\geq 1\right\} $ defined in (\ref{A89})
does not converge if $\mathbf{E}(-S_{\tau ^{-}})<\infty $ and $X$ is $(1,a) $%
-lattice with some $a\in (0,1)$.

Assume first that $a$ is rational, i.e. $a=i/j$ for some $1\leq i<j<\infty $
with g.c.d.$(i,j)=1$. Let $b=b(a)$ be the smallest natural number satisfying
$\{ab\}=1-a$. Then $\{a\left( kj+b\right) \}=1-a$ for all $k\geq 1$.
Consequently,
\begin{equation*}
\Omega (\{a\left( kj+b\right) \})=\sum_{m=0}^{\infty }H((1-a)+m)\mathbf{P}%
(X\leq -(1-a)-m)
\end{equation*}%
and
\begin{equation*}
\Omega (\{akj\})=\sum_{m=0}^{\infty }H(m)\mathbf{P}(X\leq -m).
\end{equation*}%
Observing that $\mathbf{P}(X\leq -m)=\mathbf{P}(X\leq -(1-a)-m)$, we obtain
\begin{align*}
\Omega (\{a\left( kj+b\right) \})-\Omega (\{akj\})& =\sum_{m=0}^{\infty }%
\bigl(H((1-a)+m)-H(m)\bigr)\mathbf{P}(X\leq -(1-a)-m) \\
& \geq \bigl(H(1-a)-H(0)\bigr)\mathbf{P}(X\leq -(1-a)) \\
& =H(1-a)\mathbf{P}(X<0) \\
& >\mathbf{P}(X<0).
\end{align*}%
From this inequality it follows that the sequence $\left\{ \Omega
(\{an\}),n\geq 1\right\} ,$ does not converge.

Assume now that $a$ is irrational. Define $\mathcal{N}_{1}:=\{n:\
\{an\}<(1-a)/3\}$ and $\mathcal{N}_{2}:=\{n:\ \{an\}\in (2(1-a)/3,(1-a))\}$.
The cardinality of each of the sets is infinte. In addition, one can easily
verify that
\begin{align*}
\Omega (\{an_{2}\})-\Omega (\{an_{2}\})& \geq \Bigl(H(2(1-a)/3)-H((1-a)/3)%
\Bigr)\mathbf{P}(X<0) \\
& \geq \mathbf{P}(X<0)\mathbf{P}\Bigl(\chi ^{+}\in \bigl((1-a)/3,2(1-a)/3%
\bigl)\Bigr)>0
\end{align*}%
for all $n_{1}\in \mathcal{N}_{1}$ and $n_{2}\in \mathcal{N}_{2}$.
Therefore, in the case of irrational shift the sequence $\Omega (\{an\}),$ $%
n\geq 1,$ is oscillating as well.

Theorem \ref{LadEpoch} is proved.

\begin{remark}
Analysing the proof of Theorem \ref{LadEpoch} one can see that the sequence $%
\left\{ Q_{n}^{-},n\geq 1\right\} $ in (\ref{LE}) may be written in the form%
\begin{equation*}
Q_{n}^{-}:=D(\left\{ a(n-1)\right\} ),
\end{equation*}%
where $D(x),0\leq x<1,$ is a nonnegative function and where we agree to take
$a=0$ for non-lattice distributions.
\end{remark}


\section{Discussion and concluding remarks}

We see by (\ref{weak}) that the distribution of \ $\tau ^{-}$ is completely
specified by the sequence $\{\mathbf{P}\left( \,S_{n}>0\right) ,n\geq 1\}$.
As we have mentioned in the introduction, the validity of condition (\ref%
{SpitDon}) is sufficient to reveal the asymptotic behavior of $\mathbf{P}%
(\tau ^{-}>n)$ as $n\rightarrow \infty $. Thus, \ in view of (\ref{integ1}),
informal arguments based on the plausible smoothness of $l(n)$ immediately
give the desired answer%
\begin{align*}
\mathbf{P}(\tau ^{-}& =n)=\mathbf{P}(\tau ^{-}>n-1)-\mathbf{P}(\tau ^{-}>n)
\\
& =\frac{l(n-1)}{\left( n-1\right) ^{1-\rho }}-\frac{l(n)}{n^{1-\rho }}%
\approx l(n)\left( \frac{1}{\left( n-1\right) ^{1-\rho }}-\frac{1}{n^{1-\rho
}}\right) \\
& \approx \frac{(1-\rho )l(n)}{n^{2-\rho }}\sim \frac{1-\rho }{n}\mathbf{P}%
(\tau ^{-}>n)
\end{align*}%
under the Doney condition only. In the present paper we failed to achieve
such a generality. However, it is worth mentioning that the Doney condition,
being formally weaker than the conditions of Theorem \ref{LadEpoch},
requires in the general case the knowledge of the behavior of the whole
sequence $\{\mathbf{P}\left( \,S_{n}>0\right) ,n\geq 1\},$ while the
assumptions of Theorem \ref{LadEpoch} concern a single summand only. \ Of
course, imposing a stronger condition makes our life easier and allows us to
give, in a sense, a constructive proof showing what happens in reality at
the distant moment $\tau ^{-}$ of the first jump of the random walk in
question below zero. Indeed, our arguments for the case $\left\{ 0<\alpha
<2,\ \beta <1\right\} $ demonstrate (compare (\ref{lefttail}), (\ref{aa}),
and (\ref{exact1})) that for any $x_{2}>x_{1}>0$,
\begin{align*}
\lim_{n\rightarrow \infty }\mathbf{P}(S_{n-1}& \in
(c_{n}x_{1},c_{n}x_{2}]|\tau ^{-}=n) \\
& \hspace{-1cm}=\lim_{n\rightarrow \infty }\frac{\mathbf{P}(\tau ^{-}>n-1)}{%
\mathbf{P}(\tau ^{-}=n)}\int_{x_{1}}^{x_{2}}\mathbf{P}(X<-yc_{n})\mathbf{P}%
(S_{n-1}\in c_{n}dy|\tau ^{-}>n-1) \\
& \hspace{-1cm}=\lim_{n\rightarrow \infty }\frac{\mathbf{P}(\tau
^{-}>n-1)q(2-\alpha )}{\mathbf{P}(\tau ^{-}=n)\alpha n}\int_{x_{1}}^{x_{2}}%
\frac{\mathbf{P}(X<-yc_{n})}{\mathbf{P}(X<-c_{n})}\mathbf{P}(S_{n-1}\in
c_{n}dy|\tau ^{-}>n-1) \\
& \hspace{-1cm}=\frac{q(2-\alpha )}{\alpha \left( 1-\rho \right) }%
\int_{x_{1}}^{x_{2}}\frac{\mathbf{P}(M_{\alpha ,\beta }\in dy)}{y^{\alpha }}.
\end{align*}%
In view of (\ref{IDENtity}) this means that the contribution of the
trajectories of the random walk satisfying $S_{n-1}c_{n}^{-1}\rightarrow 0$
or $S_{n-1}c_{n}^{-1}\rightarrow \infty $ as $n\rightarrow \infty $ to the
event $\left\{ \tau ^{-}=n\right\} $ is negligibly small in probability. A
"typical" trajectory looks in this case as follows: it is located over the
level zero up to moment $n-1$ with $S_{n-1}\in (\varepsilon c_{n}$ $%
,\varepsilon ^{-1}c_{n})$ for sufficiently small $\varepsilon >0$ and at
moment $\tau ^{-}=n$ the trajectory makes a big negative jump $%
X_{n}<-S_{n-1} $ of order $O(c_{n}).$

On the other hand, if $\left\{ 1<\alpha <2,\ \beta =1\right\} $ and $\mathbf{%
E}(-S_{\tau ^{-}})<\infty $, then, in the $(1,a)$-lattice case, for all $%
i\geq 0$,
\begin{equation*}
\mathbf{P}(S_{n-1}=\left\{ a(n-1)\right\} +i|\tau ^{-}=n)=\frac{H(\left\{
a(n-1)\right\} +i)\mathbf{P}(X\leq -\left\{ a(n-1)\right\} -i)}{\Omega
(\left\{ a(n-1)\right\} )}(1+o(1))
\end{equation*}%
provided that $\Omega (\left\{ a(n-1)\right\} )>0$. Since
\begin{equation*}
\sum_{i=0}^{\infty }H(\left\{ a(n-1)\right\} +i)\mathbf{P}(X\leq -\left\{
a(n-1)\right\} -i)=\Omega (\left\{ a(n-1)\right\} ),
\end{equation*}%
the main contribution to $\mathbf{P}\left( \tau ^{-}=n\right) $ is given in
this case by the trajectories located over the level zero up to moment $n-1$
with $S_{n-1}\in \lbrack 0,N]$ for sufficiently big $N$ and with not "too
big" jump $X_{n}<-S_{n-1}$ of order $O(1).$

Unfortunately, our approach to investigate the behavior of $\mathbf{P}(\tau
^{-}=n)$ in the case when $\mathbf{E}(-S_{\tau ^{-}})=\infty $ and $%
\{1<\alpha <2,\,\beta =1\}\cup \{\alpha =2,\beta =0\}$ is pure analytical
and does not allow us to extract typical trajectories without further
restrictions on the distribution of $X$. However, we can still deduce from
our proof some properties of the random walk conditioned on $\{\tau ^{-}=n\}$%
. Observe that, for any fixed $\varepsilon >0$, the trajectories with $%
S_{n-1}>\varepsilon c_{n}$ give no essential contribution to $\mathbf{P}%
(\tau ^{-}=n)$. More precisely, there exists a sequence $\delta
_{n}\rightarrow 0$ such that
\begin{equation*}
\mathbf{P}(S_{n-1}>\delta _{n}c_{n}|\tau ^{-}=n)=o(1).
\end{equation*}%
Furthermore, one can easily verify that if $\sum_{j=1}^{\infty }H(j)\mathbf{P%
}(X\leq -j)=\infty $, then for every $N\geq 1$,
\begin{equation*}
\sum_{j=1}^{N}\mathbf{P}(S_{n-1}=j;\tau ^{-}>n-1)\mathbf{P}(X\leq
-j)=o\left( \frac{l(n)}{n^{3/2}}\right) \quad \text{as }n\rightarrow \infty ,
\end{equation*}%
i.e., the contribution of the trajectories with $S_{n-1}=O(1)$ to $\mathbf{P}%
(\tau ^{-}=n)$ is negligible small. As a result we see that $%
S_{n-1}\rightarrow \infty $ but $S_{n-1}=o(c_{n})$ for all "typical"
trajectories meeting the condition $\{\tau ^{-}=n\}$. Thus, in the case $\{{%
1<\alpha <2,\beta =1\}}\cap \{\mathbf{E}(-S_{\tau ^{-}})=\infty \}$ we have
a kind of "continuous transition" between the different strategies for $%
\{\beta <1\}$ and $\{{1<\alpha <2,\beta =1\}}\cap \{\mathbf{E}(-S_{\tau
^{-}})<\infty \}$.

\vspace*{12pt} \textit{Acknowledgement}. The first version of the paper was
based on the preprint \cite{VW07}. We are thankful to an anonimous referree
who attracted our attention to the fact that by our methods one can prove
not only local theorems \ref{LadEpoch} and \ref{LadEpoch'} but the Gnedenko
and Stone type conditional local theorems \ref{NormalDev'}-\ref{SmallDev} as
well. \ \textbf{\ }V.W. is thankful to Anatoly Mogulskii for simulating
discussions on ladder epochs\textbf{.}

This project was started during the visits of the first author to the
Weierstrass Institute in Berlin and the second author to the Steklov
Mathematical Institute in Moscow. The hospitality of the both institutes is
greatly acknowledged.


\end{document}